\documentclass[letter,11pt]{article}

\usepackage[margin=1in]{geometry}
\usepackage[useregional]{datetime2}
\usepackage[english]{babel}
\usepackage[utf8]{inputenc}
\usepackage{xcolor, graphicx,changepage,xcolor, float, array, amssymb, amsmath, mathrsfs, amsfonts, hyperref, url, amsthm, tikz-cd, fancyhdr, enumerate, bbm}
\makeatletter
\newcommand*{\rom}[1]{\expandafter\@slowromancap\romannumeral #1@}
\makeatother

\makeatletter
\renewcommand*\env@matrix[1][*\c@MaxMatrixCols c]{%
	\hskip -\arraycolsep
	\let\@ifnextchar\new@ifnextchar
	\array{#1}}
\makeatother

\theoremstyle{definition}
\newtheorem{theorem}{Theorem}[section]
\newtheorem{lemma}[theorem]{Lemma}
\newtheorem{definition}[theorem]{Definition}

\newtheorem{corollary}[theorem]{Corollary}
\newtheorem{examples}[theorem]{Examples}

\newcommand{\uproman}[1]{\uppercase\expandafter{\romannumeral#1}}
\DeclareMathOperator*{\esssup}{ess\,sup}

\theoremstyle{remark}
\newtheorem{remark}[theorem]{Remark}

\usepackage{setspace}
\title{Uniform vector-valued pointwise ergodic theorems for operators}
\author{Micky Barthmann\footnote{Chemnitz University of Technology, Faculty of Mathematics,  Reichenhainer Straße 41, 09126 Chemnitz, Germany. Email: micky.barthmann@mathematik.tu-chemnitz.de}\ \ and Sohail Farhangi\footnote{University of Adam Mickiewicz, Department of Mathematics and Informatics, ulica Wieniawskiego 1, 61-712 Poznań, Poland. Email: sohail.farhangi@gmail.com}}
\begin{document}
\maketitle
\noindent\textbf{Abstract.} We prove a uniform vector-valued Wiener-Wintner Theorem for a class of operators that includes compositions of ergodic Koopman operators with contractive multiplication operators. Our results are new even in the case of complex-valued functions, as they also apply to some non-positive non-contractive operators, and they give new uniform pointwise theorems for ergodic, weakly mixing, and mildly mixing Koopman operators.% Our method of proof is to view $\xi := (T^nf(x))_{n = 1}^\infty$ as an element of a Ces\`aro sequence space, which we then map into a subspace of an ultraproduct space that possesses a natural shift operator $S$. Weighted pointwise ergodic theorems then correspond to dual pairings $\langle \xi,g^\prime\rangle = 0$, where we have, for example, that $\xi$ is weakly mixing with respect to $S$ and $g^\prime$ is compact with respect to $S^*$. 
\\

\noindent\textbf{Matematics Subject Classification.} 	Primary 37A25, 47A35; Secondary 28B05, 46M07.\\

\section{Introduction}
\subsection{Notation}
We write $(X,\mathscr{B},\mu)$ to denote a complete metric space $X$, a $\sigma$-finite Borel measure $\mu$, and we let $\mathscr{B}$ be the completion of the Borel $\sigma$-algebra with respect to $\mu$.
We use $\varphi:X\rightarrow X$ to denote a measurable transformation that preserves $\mu$, i.e., $\mu(A) = \mu(\varphi^{-1}A)$ for all $A \in \mathscr{B}$. We call the tuple $(X,\mathscr{B},\mu,\varphi)$ a measure preserving system, and we call it a probability measure preserving system if $\mu$ is a probability measure. A probability measure preserving system is ergodic if the only $A \in \mathscr{B}$ for which $\mu(A\triangle \varphi^{-1}A) = 0$ satisfy $\mu(A) \in \{0,1\}$. When our space is $X = [0,1]$, we use $m$ to denote the Lebesgue measure on $([0,1],\mathscr{B})$, where $\mathscr{B}$ is understood to be the Lebesgue $\sigma$-algebra. For a Banach space $E$ and for $1 \le p \le \infty$, we let $L^p(X,\mu;E)$ denote the Bochner space of strongly measurable functions $f:X\rightarrow E$ for which $\|f\| \in L^p(X,\mu)$, and we give more background on Bochner spaces in Section \ref{BochnerSpaceSubsection}. We use $L^p(X,\mu)$ for $L^p(X,\mu;\mathbb{C})$. We let $T_\varphi:L^p(X,\mu;E)\rightarrow L^p(X,\mu;E)$ denote the Koopman operator induced by $\varphi$, which is given by $T_\varphi f = f\circ\varphi$. We let $E^\prime$ denote the dual space of $E$. For $f \in E$ and $g^\prime \in E^\prime$, we use $\langle f,g^\prime\rangle$ and $g^\prime(f)$ interchangeably to denote the evaluation of $g^\prime$ at $f$. We use $e(x) := e^{2\pi ix}$, $\mathbb{S}^1$ denotes the unit circle in $\mathbb{C}$, and $D_1$ denotes the closed unit ball in $\mathbb{C}$. 
%%%%%%%%%%%%%%%%%%%%%%%%%%%%%%%%%%%%%%%%%%%%%%%%%%%%%%%%%%%%%%%%%%%%%%%%%%%%%%%%%%%%%%%%%%%%%%%%%%%%%%%%%%%%%%%%%%%%%%%%%%%%%%%%%
\subsection{History and statement of results}
\begin{theorem}[Birkhoff, \cite{Birkhoff'sErgodicTheorem}] Let $(X,\mathscr{B},\mu,\varphi)$ be a probability measure preserving system, and let $f \in L^1(X,\mu)$. For a.e. $x \in X$, we have
\label{BET}
\begin{equation}
\lim_{N\rightarrow\infty}\frac{1}{N}\sum_{n = 1}^Nf(\varphi^nx) = f^*(x),
\end{equation}
where $f^*(x) \in L^1(X,\mu)$ is such that $f^*(\varphi x) = f^*(x)$ for a.e. $x \in X$ and $\int_Af^*d\mu = \int_Afd\mu$ for every $A \in \mathscr{B}$ satisfying $A = \varphi^{-1}A$. In particular, if $\varphi$ is ergodic, then for a.e. $x \in X$ we have

\begin{equation}
\lim_{N\rightarrow\infty}\frac{1}{N}\sum_{n = 1}^Nf(\varphi^nx) = \int_Xfd\mu. 
\end{equation}
\end{theorem}

Birkhoff's theorem has been generalized in many directions. One direction of generalization has been to consider linear operators that are more general than Koopman operators. This program began with the work of Doob \cite{DoobPointwiseTheorem} and Kakutani \cite{BirkhoffForMarkoffProcesses} to produce a pointwise ergodic theorem for Markoff processes. Later Hopf \cite{HopfErgodicTheorem} proved a general operator theoretic pointwise ergodic theorem, which was further refined by Dunford and Schwartz \cite{DSOperators}, and then extended to operators on Bochner spaces by Chacon.

\begin{theorem}[{\cite[Theorem 1]{ChaconErgodicTheorem}}]\label{ChaconErgodicTheorem}
    Let $E$ be a reflexive Banach space, let $1 \le p < +\infty$, let $(X,\mathscr{B},\mu)$ be a $\sigma$-finite measure space, and let $T:L^1(X,\mu;E)\rightarrow L^1(X,\mu;E)$ be a linear contraction for which we also have $\|Tg\|_\infty \le \|g\|_\infty$ whenever $g \in L^1(X,\mu;E)\cap L^\infty(X,\mu;E)$. For any $f \in L^p(X,\mu;E)$, the limit

    \begin{equation}\label{EquationFromChaconsTheorem}
        \lim_{N\rightarrow\infty}\frac{1}{N}\sum_{n = 1}^NT^nf(x)
    \end{equation}
    converges in the norm topology of $E$ for a.e. $x \in X$. Furthermore, if $1 < p < +\infty$, then there exists a $f^* \in L^p(X,\mu;E)$ such that for a.e. $x \in X$ we have

    \begin{equation}
        \sup_{N\in \mathbb{N}}\left|\left|\frac{1}{N}\sum_{n = 1}^NT^nf(x)\right|\right| \le \|f^*(x)\|.
    \end{equation}
\end{theorem}

Yoshimoto \cite{MostGeneralVectorValuedErgodicTheorem} extended Chacon's result to more general operators\footnote{A special case of Yoshimoto's result is that the limit in Equation \eqref{EquationFromChaconsTheorem} still exists if $p = 1$, $\mu$ is a probability measure, and there exists some $C > 0$ for which $\sup_{n \ge 1}\|T^ng\|_\infty \le C\|g\|_\infty$ for all $g \in L^\infty(X,\mu;E)$.} and to a larger class of functions. Similar results were also found independently by Chilin and Litvinov \cite{DSValiditySpace}.\\

Another direction in which Birkhoff's theorem has been generalized is to consider measure preserving systems that have stronger mixing properties than ergodicity. The first instance of this is given by the Wiener-Wintner Theorem \cite{Wiener-WintnerTheorem}. We now state what follows from Bourgain's uniform Wiener-Wintner Theorem \cite{DoubleRecurrenceAndASConvergence}.

\begin{theorem}\label{UWWT} 
Let $(X,\mathscr{B},\mu,\varphi)$ be a probability measure preserving system and let $f \in L^1(X,\mu)$. There exists $X_f \in \mathscr{B}$ with $\mu(X_f) = 1$, such that for $x \in X_f$ and $\lambda \in \mathbb{S}^1$ we have existence of the following limit:

\begin{equation}
    \lim_{N\rightarrow\infty}\frac{1}{N}\sum_{n = 1}^Nf(\varphi^nx)\lambda^n.
\end{equation}
Furthermore, if $f \in L^1(X,\mu)$ is weakly mixing\footnote{This definition of weakly mixing function is motivated by the definitions that we use later on. Previous authors would instead say that $f$ is orthogonal to the Kronecker factor of $(X,\mathscr{B},\mu,\varphi)$, which is equivalent to our condition. It is also worth noticing that $f$ being weakly mixing is the same as $f$ being a flight vector for $T_\varphi$.}, i.e., 
\begin{equation}
    \lim_{N\rightarrow\infty}\frac{1}{N}\sum_{n = 1}^N\left|\int_XT_\varphi^nfgd\mu\right| = 0,
\end{equation}
for all $g \in L^\infty(X,\mu)$, then for $x \in X_f$ we have

\begin{equation}
\label{WWTEquation}
\lim_{N\rightarrow\infty}\sup_{\lambda \in \mathbb{S}^1}\frac{1}{N}\left|\sum_{n = 1}^Nf(\varphi^nx)\lambda^n\right| = 0. 
\end{equation}
\end{theorem}

A polynomial Wiener-Wintner Theorem was proven by Lesigne \cite{OriginalUniformPolynomialWW}, and uniform polynomial Wiener-Wintner Theorems were proven by Lesigne \cite{PolynomialWW} (cf. Theorem \ref{LesignePolynomialTheorem}), Frantzikinakis \cite{UniformPolynomialWW}, and Eisner and Krause \cite{UniformConvergenceOfTwistedErgodicAverages} (see also \cite{UniformNilsequenceWW}). Uniform Wiener-Wintner Theorems  for more general operators and function spaces were studied by Chilin, \c{C}\"{o}mez, and Litvinov \cite{AUWWforDSOperators}, \cite{MorePointwiseErgodicTheoremsForDSOperators}, and O'Brien \cite{NonCommutativeDSWienerWintnerTheorem}, but we emphasize that these latter results involve uniformity with respect to the points $x \in X$ rather than the coefficients $\lambda \in \mathbb{S}^1$. Furthermore, these authors do not make any weak mixing assumptions on their operators, so their results are not about uniform convergence to $0$. It is therefore natural to ask if a weak mixing assumption on more general operators, such as Dunford-Schwartz operators, can yield pointwise convergence to $0$. In general, this is not true as seen by considering multiplication operators. For example, the operator $M_e:L^1([0,1],m)\rightarrow L^1([0,1],m)$ given by $(M_ef)(x) = e(x)f(x)$ is weakly mixing, but we see that for any $f \in L^1([0,1],m)$ and any $x \in [0,1]$ and $\lambda_x: = e(-x)$ we have

\begin{equation}
    \lim_{N\rightarrow\infty}\frac{1}{N}\sum_{n = 1}^NM_e^nf(x)\lambda_x^n = f(x).
\end{equation}
While $M_e$ does not satisfy our desired weighted pointwise ergodic theorem, it does satisfy weighted (modulated) norm ergodic theorems with similar weights, such as \cite[Theorem 2.2]{OnModulatedErgodicTheorems}. 

One of our main results is a uniform Wiener-Wintner Theorem for weakly mixing operators, and in light of the previous example we require the following definition.

\begin{definition}
    Let $(X,\mathscr{B},\mu)$ be a $\sigma$-finite measure space, let $E$ be a Banach space, and let $T:L^1(X,\mu;E)\rightarrow L^1(X,\mu;E)$ be a bounded linear operator. The operator $T$ is \textbf{pointwise absolutely Ces\`aro bounded (paCb)} if there exists a constant $C > 0$ such that for any $f \in L^1(X,\mu,E)$ we have for a.e. $x \in X$ that

    \begin{equation}
       \limsup_{N\rightarrow\infty}\frac{1}{N}\sum_{n = 1}^N\|T^nf(x)\| \le C\int_X\|f\|d\mu.
    \end{equation}
    The operator $T$ is \textbf{strongly pointwise absolutely Ces\`aro bounded (spaCb)} if it is paCb and for any $f \in L^1(X,\mu;E)\cap L^\infty(X,\mu;E)$ we have $\displaystyle\sup_{n \in \mathbb{N}}\|T^nf\|_\infty \le C\|f\|_\infty$.
\end{definition}

In Section \ref{SubsectionOnThepaCbOperators} we study properties of (s)paCb operators and we consider various examples of operators with and without the (s)paCb property. For now, we only mention that compositions of ergodic Koopman operators with contractive multiplication operators are spaCb, in Examples \ref{ExamplesOfOperators}\eqref{MixingDoesNotImplypaCb}we verify that $M_e$ is not paCb even though it is strongly mixing, and in Examples \ref{ExamplesOfOperators}\eqref{NoncontractiveExample} we give an example of a spaCb operator $S$ that is not a positive operator and satisfies $\|S\|_1 = \|S\|_\infty = 2$.

We can now state our generalization of the uniform Wiener-Wintner Theorem, which is a special case of what follows from Lemma \ref{LemmaUsingNiceness} and Theorems \ref{firstapplicationkeylemma}(ii) and \ref{UniformConvergenceForNWMSequences}.

\begin{theorem}\label{UniformWienerWintnerInIntro}
    Let $(X,\mathscr{B},\mu)$ be a $\sigma$-finite measure space, let $E$ be a Banach space, and let $T:L^1(X,\mu;E)\rightarrow L^1(X,\mu;E)$ be a bounded linear spaCb operator. Then for any weakly mixing $f \in L^1(X,\mu;E)$, i.e., any $f$ satisfying
    
    \begin{equation}
        \lim_{N\rightarrow\infty}\frac{1}{N}\sum_{n = 1}^N\left|\langle T^nf,g^\prime\rangle\right| = 0,
    \end{equation}
    for all $g^\prime \in L^1(X,\mu;E)^\prime$, we have for a.e. $x \in X$ that

    \begin{equation}\label{OurWWEquationIntro}
        \lim_{N\rightarrow\infty}\sup_{\lambda \in \mathbb{S}^1}\left|\left|\frac{1}{N}\sum_{n = 1}^NT^nf(x)\lambda^n\right|\right| = 0.
    \end{equation}
\end{theorem}

Our method of proof is as follows. We first view the sequence $\xi := (T^nf(x))_{n = 1}^\infty$ as a vector in a certain subspace of an ultraproduct that we denote by $A_\mathcal{U}(E)$, and the spaCb property of $T$ ensures that $\|\xi\|_\mathcal{U} \le C\|f\|_1$. We then show that the mixing properties of $f$ with respect to $T$ translate into mixing properties of the vector $\xi$ with respect to a left shift operator $S_\mathcal{U}$ on $A_\mathcal{U}(E)$. We then show that weighted pointwise ergodic theorems correspond to dual pairings $\langle\xi,g^\prime\rangle$ for some $g^\prime \in A_\mathcal{U}(E)^\prime$. In particular, if $\xi$ is weakly mixing with respect to $S_\mathcal{U}$, and $g^\prime$ is compact (reversible) with respect to the adjoint operator $S_\mathcal{U}^*$, then $\langle \xi,g^\prime\rangle = 0$. Lastly, we argue by contradiction that if the convergence in Equation \eqref{OurWWEquationIntro} was not uniform with respect to $\lambda$, then we could take a sequence of counterexamples that converge to some compact (reversible) functional $g^\prime$ satisfying  $\langle \xi,g^\prime\rangle \neq 0$.

The method described in the previous paragraph can be seen as a generalization of the methods used in \cite{SohailsFirstPointwiseErgodicTheorem} and \cite[Chapter 3]{SohailsPhDThesis}. This method is also flexible enough to yield uniform pointwise ergodic theorems for ergodic operators/functions as well as mildly mixing operators/functions (cf. Theorem \ref{UniformConvergenceForNMMSequences}), both of which are new even in the case of Koopman operators. Our next result is a consequence of Lemma \ref{LemmaUsingNiceness} and Theorems \ref{firstapplicationkeylemma}(i) and \ref{UniformConvergenceForCESequences}.

\begin{theorem}\label{UniformBirkhoffInIntro}
    Let $(X,\mathscr{B},\mu)$ be a $\sigma$-finite measure space, let $E$ be a Banach space, and let $T:L^1(X,\mu;E)\rightarrow L^1(X,\mu;E)$ be a bounded linear spaCb operator. Then for any sequence $(\delta_n)_{n = 1}^\infty \subseteq \mathbb{R}^+$ satisfying $\lim_{n\rightarrow\infty}\delta_n = 0$, any ergodic $f \in L^1(X,\mu;E)$, i.e., any $f$ satisfying
    
    \begin{equation}
        \lim_{N\rightarrow\infty}\frac{1}{N}\sum_{n = 1}^N\langle T^nf,g^\prime\rangle = 0,
    \end{equation}
    for all $g^\prime \in L^1(X,\mu;E)^\prime$, we have for a.e. $x \in X$ that

    \begin{alignat*}{2}
        &\lim_{N\rightarrow\infty}\sup_{(c_n)_{n = 1}^N \in \mathcal{I}(N,\delta_N)}\left|\left|\frac{1}{N}\sum_{n = 1}^NT^nf(x)c_n\right|\right| = 0\text{, where }\\
        &\mathcal{I}(N,\delta_N) := \left\{(c_n)_{n = 1}^N \in D_1^N\ |\ \frac{1}{N}\sum_{n = 1}^{N-1}|c_n-c_{n+1}| < \delta_N\right\}.
    \end{alignat*}
    
\end{theorem}
The structure of the paper is as follows. In Section 2 we record facts about Bochner spaces, filters and ultrafilters, mixing properties of operators, the spaCb property, and ultraproducts for use in future sections. We also show in Section 2.6 that the spaCb property is not strong enough to yield polynomial pointwise ergodic theorems. In Section 3 we prove our main results. In Section 4 we review various occurrences of mixing sequencees throughout the literature to help give context to the mixing sequences we define here.

%\noindent\textbf{Acknowledgements:} The second author acknowledges being supported by grant
%2019/34/E/ST1/00082 for the project “Set theoretic methods in dynamics and number theory,” NCN (The
%National Science Centre of Poland), and the first author was supported by the same grant for a 3 week research visit.\\
%The authors would like to thank Noa Bihlmaier for good discussions. 
%%%%%%%%%%%%%%%%%%%%%%%%%%%%%%%%%%%%%%%%%%%%%%%%%%%%%%%%%%%%%%%%%%%%%%%%%%%%%%%%%%%%%%%%%%%%%%%%%%%%%%%%%%%%%%%%%%%%%%%%

\section{Preliminaries}
\subsection{Bochner spaces}\label{BochnerSpaceSubsection}
Let $(X,\mathscr{B},\mu)$ be a $\sigma$-finite measure space and let $E$ be a Banach space. A function $s:X\rightarrow E$ is called $\mu$\textbf{-simple function} if $s=\sum_{n=1}^N\xi_n1_{A_n}$ for some $N\in\mathbb{N}$, where $\xi_n\in E$ and $A_n\in \mathscr{B}$ with $\mu(A_n)<\infty$. For a $\mu$-simple function $s=\sum_{n=1}^N\xi_n1_{A_n}$ we define 
\[\int_X s d\mu = \sum_{n=1}^N \mu(A_n)\xi_n.\]
A function $f:X\rightarrow E$ is called \textbf{strongly }$\mu$-\textbf{measurable} if the exists a sequence $(s_n)_{n=1}^\infty$ of $\mu$-simple functions converging to $f$ $\mu$-almost everywhere, which is equivalent to $f$ having a separable range. Moreover a strongly $\mu$-measurable $f:X\rightarrow E$ is called \textbf{Bochner integrable} if there exists a sequence $(s_n)_{n=1}^\infty$ of $\mu$-simple functions such as 
\[\lim_{n\rightarrow\infty}\int_X \|s_n - f\| d\mu =0.\]
This gives the notion of the \textbf{Bochner integral} for a Bochner integrable $f$, defined via
\[\int_X f d\mu =\lim_{n\rightarrow\infty}\int_X s_n d\mu.\]
Notice that a strongly $\mu$-measurable $f:X\rightarrow E$ is Bochner integrable if and only if 
\[\int_X \|f\|d\mu <\infty.\]
Being equal $\mu$-almost everywhere gives an equivalence relation $\sim_{\mu}
$ on the set of all strongly $\mu$-measurable functions from $X$ to $E$. For $1\leq p <\infty$ the vector space $L^p(X,\mu;E)$ is defined as set of all equivalence classes of strongly $\mu$-measurable functions $f:X\rightarrow E$ such that
\[\|f\|_{L^p_E}:=\left(\int_X \|f\|^p d\mu\right)^{1/p} <\infty.\]
Moreover the vector space $L^\infty(X,\mu;E)$ is defined as set of all equivalence classes of strongly $\mu$-measurable functions $f:X\rightarrow E$ such that
\[\|f\|_{L^\infty_E}:=\esssup_{x\in X}\|f(x)\|<\infty.\]
For any $p\in[1,\infty]$ the spaces $(L^p(X,\mu;E),\|\cdot\|_{L^p_E})$ are complete. The following theorem is a generalization of Birkhoff's Pointwise Ergodic Theorem for Bochner spaces.
\begin{theorem}[{\cite[Theorem 4.2.1]{UKErgodicThms}}]\label{BirkhoffForBochnerSpaces}
    Let $(X,\mathscr{B},\mu,\varphi)$ be a $\sigma$-finite measure preserving system, let $E$ be a Banach space, and let $f \in L^1(X,\mu;E)$. Then

    \begin{equation}
        \lim_{N\rightarrow\infty}\frac{1}{N}\sum_{n = 1}^Nf(\varphi^nx) = f^*(x),
    \end{equation}
    where convergence takes place pointwise in the strong topology of $E$, $f^* \in L^1(X,\mu;E)$, and $T_\varphi f^* = f^*$.
\end{theorem}
It is worth mentioning that Theorem \ref{BirkhoffForBochnerSpaces} is quickly deduced from the simpler case in which $E 
= \mathbb{C}$. 
Similarly, if our operator is a Koopman operator $T_\varphi$, then weighted pointwise ergodic theorems for Bochner spaces follow easily from the corresponding theorems for complex-valued functions. 
The main interest in working with Bochner spaces is when we do not work with Koopman operators, as there are more bounded linear operators on $L^1(X,\mu;E)$ than there are on $L^1(X,\mu)$. For example, it is shown in \cite[Theorem 2.1.3]{AnalysisInBanachSpacesVolume1} that if $T:L^1(X,\mu)\rightarrow L^1(X,\mu)$ is a bounded linear positive operator, then it has a unique natural extension to a bounded linear operator  $\tilde{T}:L^1(X,\mu;E)\rightarrow L^1(X,\mu;E)$. We remark that the operators appearing in Examples \ref{ExamplesOfOperators} \eqref{MultiplicationOperatorsExamples} and \eqref{ErgodicComposedWithMultiplicationIsspaCb} need not arise in this fashion.

Since the map 
\begin{align*}
    &J:L^q(X,\mu)\rightarrow L^p(X,\mu)^\prime, g\mapsto J_g; \\
    &\text{with }J_g(f)=\int_X f(x)g(x)d\mu
\end{align*}
    is an isometric isomorphsim of Banach spaces for $\frac{1}{p}+\frac{1}{q} = 1$, one can ask the same for Bochner spaces. To this end, we record the following result that is a special case of Theorems 1.3.10 and 1.3.21 of \cite{AnalysisInBanachSpacesVolume1}.
    \begin{theorem}\label{ReflexivitätBochner}
Let $(X,\mathscr{B},\mu)$ be a $\sigma$-finite measure space, let $E$ be a reflexive Banach space, and let $1\leq p <\infty$. Then for $q$ satisfying $\frac{1}{p}+\frac{1}{q} = 1$, the map 
  \begin{align*}
    &I:L^q(X,\mu;E^\prime)\rightarrow L^p(X,\mu;E)^\prime, g\mapsto I_{g}; \\
    &\text{with }I_g(f)=\int_X \langle f(x),g(x)\rangle d\mu(x)
\end{align*}
is an isometric isomorphism. In particular, $L^p(X,\mu;E)$ is reflexive for all $1<p<\infty$.
    \end{theorem}
    
%%%%%%%%%%%%%%%%%%%%%%%%%%%%%%%%%%%%%%%%%%%%%%%%%%%%%%%%%%%%%%%%%%
\subsection{Filters and Ultrafilters}
\begin{definition}
    A collection of subsets $\mathcal{U}\subseteq \mathcal{P}(\mathbb{N})$ is called \textbf{filter} if 
    \begin{enumerate}
        \item[(i)] $\emptyset\notin \mathcal{U}$ and $\mathbb{N}\in \mathcal{U}$.
        \item[(ii)] If $A,B\in \mathcal{U}$ then $A\cap B\in \mathcal{U}$. 
        \item[(iii)] If $A\in \mathcal{U}$ and $A\subseteq B$ then $B\in \mathcal{U}$.
    \end{enumerate}
    The collection $\mathcal{U}$ is called an \textbf{ultrafilter} if it is a filter and for any $A \subseteq \mathbb{N}$ we have the either $A \in \mathcal{U}$ or $A^c \in \mathcal{U}$.
\end{definition}
\begin{examples}
\phantom{asd}
    \begin{enumerate} 
        \item[(i)] Let $A\subseteq\mathbb{N}$ and $A_n=\{a\in A: a\leq n\}$ for all $n\in\mathbb{N}$. Then the natural density of $A$ is given by 
        \[d(A)=\lim_{n\rightarrow\infty}\frac{|A_n|}{n},\]
        provided it exists. We call the collection $\mathcal{D}$ of sets $A\subseteq\mathbb{N}$ with $d(A)=1$ the \textbf{density one filter}.
        \item[(ii)] Let $\mathcal{P}_f(\mathbb{N})$ denote the collection of finite subsets of $\mathbb{N}$. For a given sequence $(n_k)_{k\in\mathbb{N}}\subseteq \mathbb{N}$ we denote by $FS((n_k)_{k\in\mathbb{N}}):=\{\sum_{k\in A}n_k\}_{A\in \mathcal{P}_f(\mathbb{N})}$ the set of finite sums of $(n_k)_{k\in\mathbb{N}}$. A set $A\subseteq \mathbb{N}$ is called \textbf{IP-set} if there exists some $(n_k)_{k\in\mathbb{N}}$ such that $ FS((n_k)_{k\in\mathbb{N}})\subseteq A$. A set $B\subseteq \mathbb{N}$ is called \textbf{IP*-set} if $A\cap B \neq \emptyset$ whenever $A$ is an IP-set. We call the collection $\mathcal{IP}^*$ of IP*-sets the \textbf{IP*-filter}. 
        \item[(iii)] We call the collection of sets $\mathcal{P}_c:=\{A\subseteq\mathbb{N}: A^c\text{ is finite}\}$ the \textbf{cofinite filter}. The cofinite filter is also commonly known as the \textbf{Fr\'echet filter}.
        \item[(iv)] For any $n \in \mathbb{N}$, the collection $\mathcal{U}_n := \{A \subseteq \mathbb{N}\ |\ n \in A\}$ is an ultrafilter. Ultrafilters of the form $\mathcal{U}_n$ are \textbf{principal}, and all other ultrafilters are \textbf{nonprincipal} (or 
        \textbf{free}). Nonprincipal ultrafilters necessarily contain the cofinite filter.
    \end{enumerate}
\end{examples}

The density one filter, the IP*-filter, and the cofinite filter are not ultrafilters. We denote the collection of ultrafilters by $\beta\mathbb{N}$ and denote the collection of nonprincipal ultrafilters by $\beta\mathbb{N}^*$. 
\begin{definition}
    Let X be a Hausdorff space and $\mathcal{U}$ a filter. We call $x\in X$ the $\mathcal{U}$-\textbf{limit} of a sequence $(x_n)_{n = 1}^\infty\subseteq X$ if for every open neighbourhood $V$ of $x$,
    \[\{n\in\mathbb{N}:x_n\in V\}\in \mathcal{U}\]
    holds. We denote 
    \[\mathcal{U}-\lim_{n\rightarrow\infty}x_n = x\]
    if $x$ is the $\mathcal{U}$-limit of $(x_n)_{n = 1}^\infty$. If $X$ is compact and $\mathcal{U}$ is an ultrafilter, then $\mathcal{U}-\lim_{n\rightarrow\infty}x_n$ exists for any sequence $(x_n)_{n = 1}^\infty \subseteq X$.
\end{definition}

The next lemma can intuitively be viewed as a statement that the pullback $f^*\mathcal{U}$ of an ultrafilter $\mathcal{U}$ under a function $f:\mathbb{N}\rightarrow\mathbb{N}$ satisfying $f^{-1}(\mathbb{N}) \in \mathcal{U}$, is again an ultrafilter.

\begin{lemma}\label{PushforwardOfUltrafilters}
    Let $\mathcal{U} \in \beta\mathbb{N}$ be arbitrary and let $\{N_m : m \in \mathbb{N}\} \in \mathcal{U}$ be such that $N_{m+1} > N_m$. The collection $\mathcal{V} := \{B\ |\ \{N_m : m \in B\} \in \mathcal{U}\}$ is an ultrafilter. Furthermore, $\mathcal{V}$ is nonprincipal if and only if $\mathcal{U}$ is nonprincipal.
\end{lemma}

\begin{proof}
    It is clear that $\mathbb{N} \in \mathcal{V}$ and $\emptyset \notin \mathcal{V}$. It is also clear that if $A \in \mathcal{V}$ and $A \subseteq B$, then $B \in \mathcal{V}$.
    Next, we see that if $B_1,B_2 \in \mathcal{V}$, then $\{N_m : m \in B_1\cap B_2\} = \{N_m : m \in B_1\}\cap\{N_m : m \in B_2\}$, so $B_1\cap B_2 \in \mathcal{V}$. Lastly, we see that for any $B \subseteq \mathbb{N}$ we have $\{N_m : m \in B\}\cup\{N_m : m \in B^c\} \in \mathcal{U}$, hence one of $B$ or $B^c$ must be in $\mathcal{V}$. Now that we have shown that $\mathcal{V}$ is an ultrafilter, the latter claim is immediate.
\end{proof}
\subsection{Mixing properties of operators}
%\noindent At this point we will give a small introduction to mean ergodicity and mixing properties of operators on Banach spaces.
Let $E$ be a Banach space and $T$ a  bounded linear operator on $E$. We call $T$ \textbf{weakly almost periodic} if for any $\xi \in E$ the set $\{T^n\xi \ |\ n\in\mathbb{N}\}$ is weakly pre-compact. The operator $T$ is called \textbf{mean ergodic} if
    \[\lim_{N\rightarrow\infty}\frac{1}{N}\sum_{n=1}^N T^n \xi\]
    converges in norm for all $\xi\in E$. The convergence induces a decomposition of the space as $E = \text{fix}(T)\oplus\overline{(I-T)E}$, where $\text{fix}(T) := \{\xi \in E\ |\ T\xi = \xi\}$. It is well known that any power bounded operator on a reflexive Banach space is weakly almost periodic, and that any weakly almost periodic operator on a Banach space is mean ergodic (see \cite[Chapter 8.4]{OTAoET}). Moreover it is  worthwhile mentioning that if $T$ a positive and  power-bounded  operator on a scalar-valued $L^1$-space, then mean ergodicity is equivalent to weak almost periodicity (see \cite{OneParameterOperatorSemigroups}, Theorem 3.1.11).

    Another classical decomposition that is satisfied by weakly almost periodic operators is the Jacobs-deLeeuw-Glicksberg Decomposition (see \cite[Chapter 16.3]{OTAoET}), which says that $E =E_{rev}\oplus E_{aws}$, where
    \begin{alignat*}{2}
        %E &=E_{aws}\oplus E_{rev}\text{, where}\\
        E_{rev} &=\overline{\text{lin}_{\mathbb{C}}}\{\xi \in E\ |\ T\xi = \lambda\xi\text{ for some }\lambda \in \mathbb{S}^1\},\\
        E_{aws}&=\{\xi \in E\ |\ T^{n_j}\xi\overset{w}{\rightarrow}0\text{ for some }(n_j)_{j = 1}^\infty \subseteq \mathbb{N}\}\\
        &= \left\{\xi \in E\ |\ \lim_{N\rightarrow\infty}\frac{1}{N}\sum_{n = 1}^N|\langle T^n\xi,g^\prime\rangle| = 0\ \forall\ g^\prime \in E^\prime\right\}.
    \end{alignat*}
 The spaces $E_{rev}$ and $E_{aws}$ are known as the reversible part and the almost weakly stable part of $E$ respectively.
\begin{definition}
    Let $E$ be a Banach space, let $\xi \in E$, and let $T$ be a bounded linear operator on $E$. We call $\xi$
    \begin{enumerate}[(i)]
    \item \textbf{ergodic} (with respect to $T$) if for any $g^\prime \in E^\prime$ we have

    \begin{equation}
        \lim_{N\rightarrow\infty}\frac{1}{N}\sum_{n = 1}^N\langle T^n\xi,g^\prime\rangle = 0.
    \end{equation}

    \item  \textbf{weakly mixing} (with respect to $T$) if for any $g^\prime \in E^\prime$ we have

    \begin{equation}
        \lim_{N\rightarrow\infty}\frac{1}{N}\sum_{n = 1}^N\left|\langle T^n\xi,g^\prime\rangle\right| = 0.
    \end{equation}

    \item  \textbf{mildly mixing} (with respect to $T$) if for any $g^\prime \in E^\prime$ we have

    \begin{equation}
        \mathcal{IP}^*-\lim_{n\rightarrow\infty}\langle T^n\xi,g^\prime\rangle = 0.
    \end{equation}

    \item  \textbf{strongly mixing} (with respect to $T$) if for any $g^\prime \in E^\prime$ we have

    \begin{equation}
        \lim_{n\rightarrow\infty}\langle T^n\xi,g^\prime\rangle = 0.
    \end{equation}

    The operator $T$ is \textbf{fully mean ergodic / weakly mixing / mildly mixing/ strongly mixing} if every $\eta \in E$ is ergodic / weakly mixing / mildly mixing / strongly mixing with respect to $T$.
    \end{enumerate}
    \begin{remark}
A weakly mixing element $\xi\in E$ with respect to a power bounded operator $T$ can be defined equivalently via
\[\mathcal{D}-\lim_{n\rightarrow\infty} \langle T^n \xi, g^\prime\rangle =0\]
for all $g^\prime\in E^\prime$ (see \cite[Theorem 9.15]{OTAoET}). A characterisation and more details about mixing properties of operators can be found in \cite[Chapter \uproman{2}]{stability}.

    \end{remark}
    
\end{definition}
\begin{theorem}\label{koopmanism}
    Let $\mathcal{X} := (X,\mathscr{B},\mu,\varphi)$ be a measure preserving system and $T_\varphi$ the Koopman operator on $L^1(X,\mu)$. Then the following assertions hold.
    \begin{enumerate}
        \item[(i)] The system $\mathcal{X}$ is ergodic if and only if every $f\in L^1(X,\mu)$ with $\int_X fd\mu=0$ is ergodic with respect to $T_\varphi$.
        \item[(ii)] The system $\mathcal{X}$ is weakly mixing if and only if every $f\in L^1(X,\mu)$ with $\int_X fd\mu=0$ is weakly mixing with respect to $T_\varphi$.
        \item[(iii)] The system $\mathcal{X}$ is mildly mixing if and only if every $f\in L^1(X,\mu)$ with $\int_X fd\mu=0$ is mildly mixing with respect to $T_\varphi$.
        \item[(iv)] The system $\mathcal{X}$ is strongly mixing if and only if every $f\in L^1(X,\mu)$ with $\int_X fd\mu=0$ is strongly mixing with respect to $T_\varphi$.
    \end{enumerate}
\end{theorem}
\noindent A proof of the statements (i),(ii), (iv) of Theorem \ref{koopmanism} can be found \cite[Chapters 8,9]{OTAoET}. The following theorem gives a generalization of Theorem \ref{koopmanism} for the vector valued case.% A proof of statement (iii) can be found in \textcolor{red}{????}
\begin{theorem}\label{meanergodic}
    Let $(X,\mathscr{B},\mu,\varphi)$ be a measure preserving system and $E$ a Banach space. Then the Koopman operator $T_\varphi$ is mean ergodic on $L^1(X,\mu;E)$. In particular if $(X,\mathscr{B},\mu,\varphi)$ is ergodic / weakly mixing / mildly mixing / strongly mixing and $\int_X f d\mu = 0$ holds, then $f$ is ergodic / weakly mixing / mildly mixing / strongly mixing with respect to $T_\varphi$.
    \begin{proof}
Since $T_\varphi$ is mean ergodic on $L^1(X,\mu)$, for all $\xi\in E$ and $B\in \mathscr{B}$ we see that
   \[\|\frac{1}{H}\sum_{h=1}^H T_\varphi^h(\mathbbm{1}_B\xi)-
   \mu(B)\xi\|_{L^1_E}=\|\frac{1}{H}\sum_{h=1}^H T_\varphi^h \mathbbm{1}_B-
   \mu(B) \mathbbm{1}\|_{L^1}\|\xi\|\ \rightarrow 0.\]
   Hence  
   \[\lim_{H\rightarrow\infty}\|\frac{1}{H}\sum_{h=1}^H T_\varphi^h f-
   (\int_Xfd\mu)\cdot \mathbbm{1} \|_{L^1_E}=0\]
   follows for any $\mu$-simple function $f\in L^1(X,\mu;E)$. Since $T_\varphi$ is power bounded on $L^1(X,\mu;E)$  Banach–Steinhaus theorem yields the first assertion. If $(X,\mathscr{B},\mu,\varphi)$ is weakly mixing / mildly mixing / strongly mixing, let $p = \mathcal{D}/\mathcal{IP}^*/\mathcal{P}_c$. For given $\xi\in E$, $g^\prime\in (L^1(X,\mu;E))^\prime$ and $f\in L^1(X,\mu)$ we define 
        \[(f\xi)(x):=f(x)\xi \quad \text{and}\quad g^\prime_\xi(f)=\langle f\xi, g^\prime\rangle.\]
    Hence $f\xi \in L^1(X,\mu;E)$, $g^\prime_\xi\in (L^1(X,\mu))^\prime$ and 
        \[p-\lim_n |\langle T^n_\varphi \mathbbm{1}_A \xi - \mu(A)\xi,g^\prime \rangle|=p-\lim_n |\langle T^n_\varphi \mathbbm{1}_A - \mu(A) ,g_\xi^\prime\rangle| =0.\]
        Since the set of $\mu$-simple functions is dense in $L^1(X,\mu;E)$ the desired results follow.
        
    \end{proof}

\end{theorem}
%%%%%%%%%%%%%%%%%%%%%%%%%%%%%%%%%%%%%%%%%%%%%%%%%%%%%%%%%%%%%%%%%%%%%%%%%%%%%%%%%%%%%%%%%%%%%%%%%%%%%%%%%
\subsection{Strongly Pointwise absolutely Ces\`aro bounded (spaCb) operators}\label{SubsectionOnThepaCbOperators}

We begin with two results that will not be used later on in this paper, but they show that the spaCb property is strong enough to replace the maximal inequalities that are usually needed in the study of pointwise ergodic theorems.

\begin{theorem}\label{RelevanceOfpaCb}
    Let $E$ be a Banach space, let $(X,\mathscr{B},\mu)$ be a $\sigma$-finite measure space, and let $(E_2,\|\cdot\|_{E_2})$ be a Banach space for which $E_2 \subseteq L^1(X,\mu;E)$ and $\|\xi\|_1 \le \|\xi\|_{E_2}$ for all $\xi \in E_2$. Let $T:L^1(X,\mu;E)\rightarrow L^1(X,\mu;E)$ be a bounded linear spaCb operator such that $T(E_2) \subseteq E_2$ and the restriction of $T$ to $E_2$ is weakly almost periodic\footnote{We remark that $T:E_2\rightarrow E_2$ is a bounded operator due to the closed graph theorem.}. For any $f \in E_2$ there exists $X_f \in \mathscr{B}$ with $\mu(X\setminus X_f) = 0$ such that for any $x \in X_f$ and any $\lambda \in \mathbb{S}^1$, the following limit exists in the strong topology of $E$:

    \begin{equation}
        \lim_{N\rightarrow\infty}\frac{1}{N}\sum_{n = 1}^NT^nf(x)\lambda^n.
    \end{equation}
\end{theorem}

\begin{proof}
    Since $T:E_2\rightarrow E_2$ is weakly almost periodic, we may use the Jacobs-deLeeuw-Glicksberg decomposition to write $f = f_1+f_2$ with $f_1 \in E_{rev}$ and $f_2 \in E_{aws}$. We see that if $g^\prime \in L^1(X,\mu;E)^\prime$ and $\xi \in E_2$, then
    
    \begin{equation}
         |\langle\xi,g^\prime\rangle| \le \|\xi\|_1\|g^\prime\| \le \|\xi\|_{E_2}\|g^\prime\|,
    \end{equation}
    so $g^\prime$ restricted to $E_2$ is an element of $E_2^\prime$. In particular, we see that

    \begin{equation}
        \lim_{N\rightarrow\infty}\frac{1}{N}\sum_{n = 1}^N|\langle T^nf_2,g^\prime\rangle| = 0,
    \end{equation}
    so $f_2 \in L^1(X,\mu;E)$ is weakly mixing. Theorem \ref{UniformWienerWintnerInIntro} tells us there exists $X_f' \in \mathscr{B}$ with $\mu(X\setminus X_f') = 0$, such that for all $x \in X_f'$ and all $\lambda \in \mathbb{S}^1$ we have
\vskip -3mm
    \begin{equation}
        \lim_{N\rightarrow\infty}\frac{1}{N}\sum_{n = 1}^NT^nf_2(x)\lambda^n = 0.
    \end{equation}
    For each $k \in \mathbb{N}$, there exists $g_k \in E_2$ that is a finite linear combination of eigenfunctions of $T$ and satisfies $\|f_1-g_k\|_1 \le \|f_1-g_k\|_{E_2}< \frac{1}{k}$. Recalling that $T$ is spaCb, we see that there exists $X_k \in \mathscr{B}$ with $\mu(X\setminus X_k) = 0$ such that for all $x \in X_k$ and all $\lambda \in \mathbb{S}^1$ we have
\vskip -7mm
    \begin{equation}
        \lim_{N\rightarrow\infty}\frac{1}{N}\sum_{n = 1}^NT^ng_k(x)\lambda^n\text{ exists, and }\limsup_{N\rightarrow\infty}\frac{1}{N}\sum_{n = 1}^N||T^n(f_1-g_k)(x)|| \le C||f_1-g_k||_1 < \frac{C}{k}.
    \end{equation}
    It follows that $X_f := X_f'\cap\bigcap_{k = 1}^\infty X_k$ has full measure, and that for all $x \in X_f$ and $\lambda \in \mathbb{S}^1$ we have
    \begin{equation}
\lim_{N\rightarrow\infty}\frac{1}{N}\sum_{n = 1}^NT^nf_1(x)\lambda^n\text{ exists, hence}\lim_{N\rightarrow\infty}\frac{1}{N}\sum_{n = 1}^NT^nf(x)\lambda^n\text{ exists.}
    \end{equation}
\end{proof}

\begin{corollary}
    Let $E$ be a reflexive Banach space, let $(X,\mathscr{B},\mu)$ be a probability space, and let $1 < p < \infty$ be arbitrary. Let $T:L^1(X,\mu;E)\rightarrow L^1(X,\mu;E)$ be a bounded linear spaCb operator such that $T$ maps $L^p(X,\mu;E)$ into itself and the restriction of $T$ to $L^p(X,\mu;E)$ is power bounded. For any $f \in L^p(X,\mu;E)$, there exists $X_f \in \mathscr{B}$ with $\mu(X\setminus X_f) = 0$ such that for any $x \in X_f$ and any $\lambda \in \mathbb{S}^1$, the following limit exists in the strong topology of $E$:

    \begin{equation}
        \lim_{N\rightarrow\infty}\frac{1}{N}\sum_{n = 1}^NT^nf(x)\lambda^n.
    \end{equation}
\end{corollary}

\begin{proof}
    Theorem \ref{ReflexivitätBochner} tells us that $L^p(X,\mu;E)$ is reflexive. Since $T:L^p(X,\mu;E)\rightarrow L^p(X,\mu;E)$ is power bounded, we see that $T$ is a weakly almost periodic operator on $L^p(X,\mu;E)$, so the desired result now follows from Theorem \ref{RelevanceOfpaCb}.
\end{proof}

\begin{examples}\label{ExamplesOfOperators}
    We will now look at examples of operators with and without the (s)paCb property. In these examples $(X,\mathscr{B},\mu)$ will be a standard probability space and $E$ a reflexive Banach space. In light of Chacon's Theorem, we will mostly focus on examples that are contractions on $L^1(X,\mu;E)$ and $L^\infty(X,\mu;E)$, which we call $L^1-L^\infty$ contractions. 
    \begin{enumerate}[(i)]
        \item The Koopman operator of a measure preserving transformation, or more generally, of a Markoff process, is a $L^1-L^\infty$ contraction. If the measure preserving transformation is ergodic, or the Markoff process is irreducible, then the corresponding Koopman operator is spaCb as a consequence of Birkhoff's pointwise ergodic theorem in the former case and Hopf's pointwise ergodic theorem in the latter.
        
        \item\label{MultiplicationOperatorsExamples} If $f \in L^\infty(X,\mu)$ satisfies $\|f\|_\infty \le 1$, then multiplication operator $M_f(g)(x) := f(x)g(x)$ is a $L^1(X,\mu)-L^\infty(X,\mu)$ contraction. More generally, if $\mathcal{L}_1(E)$ denotes the set of linear operators $T$ on $E$ satisfying $\|T\| \le 1$, then for any $f:X\rightarrow \mathcal{L}_1(E)$ that is strongly measurable with respect to the Borel $\sigma$-algebra of the strong operator topology\footnote{We use this $\sigma$-algebra because \cite[Proposition 1.1.28]{AnalysisInBanachSpacesVolume1} tells us in this case that if $g:X\rightarrow E$ is strongly measurable, then $M_f(g):X\rightarrow E$ is also strongly measurable.}, the operator $M_f(g)(x) = f(x)(g(x))$ is a $L^1(X,\mu;E)-L^\infty(X,\mu;E)$ contraction. However, $M_f$ will not be a paCb operator if the range of $f$ consists of isometries. To see this, let $(X,\mathscr{B},\mu)$ be non-atomic, let $C > 0$ be arbitrary and take any $A \in \mathscr{B}$ with $0 < \mu(A) < \frac{1}{C}$ and any $\xi \in E$ and consider $g = \xi\mathbbm{1}_A$. We see that for any $x \in A$, we have

        \begin{equation}
            C\int_X\|g\|d\mu = C\|\xi\|\mu(A) < \|\xi\| = \lim_{N\rightarrow\infty}\frac{1}{N}\sum_{n = 1}^N\|g(x)\| = \lim_{N\rightarrow\infty}\frac{1}{N}\sum_{n = 1}^N\|(M_f^ng)(x)\|
        \end{equation}
        If $r(f(x))<1$ holds for a.e. $x\in X$ (where $r$ denotes the spectral radius) then $M_f$ is paCb. This follows directly since $\displaystyle\lim_{n\rightarrow\infty}\|(f(x))^n\|=0$
        holds for a.e. $x\in X$. 
        %\item Let $(Y,\mathscr{A},\nu)$ be a probability space, let $\mathcal{C}$ denote the space of $L^1-L^\infty$ contractions, and let $f:Y\rightarrow \mathcal{C}$ be measurable. Then $\int_Yfd\nu \in \mathcal{C}$, so in particular, $\mathcal{C}$ is closed under convex combinations%I commented this out since we have no idea how convex combinations interact with the paCb property
        
        \item A composition of $L^1-L^\infty$ contractions is once again a $L^1-L^\infty$ contraction. However, the (s)paCb property need not be preserved by composition. To see this, let $\alpha \in \mathbb{R}\setminus\mathbb{Q}$  be arbitrary and let $\varphi:[0,1]\rightarrow[0,1]$ be given by $\varphi(x) = x+\alpha\pmod{1}$. Since $\varphi$ and $\varphi^{-1}$ are ergodic measure preserving transformations, we see that $T_\varphi$ and $T_{\varphi^{-1}}$ are (s)paCb, but $\text{Id} = T_\varphi\circ T_{\varphi^{-1}}$ is not (s)paCb.

        \item\label{ErgodicComposedWithMultiplicationIsspaCb} Let $f:X\rightarrow\mathcal{L}_1(E)$ be strongly measurable and let $\varphi:X\rightarrow X$ be an ergodic measure preserving transformation. The operator $M_fT_\varphi$ is spaCb, because for any $g \in L^1(X,\mu;E)$ we have

        \begin{equation}
            \limsup_{N\rightarrow\infty}\frac{1}{N}\sum_{n = 1}^N\|((M_fT_\varphi)^ng)(x)\| \le \lim_{N\rightarrow\infty}\frac{1}{N}\sum_{n = 1}^N\|g(\varphi^nx)\| = \int_X\|g\|d\mu,
        \end{equation}
        for a.e. $x \in X$, where the final equality follows from Birkhoff's pointwise ergodic theorem. In particular the operator $S_{f,\varphi}:L^1(X,\mu;E)\rightarrow L^1(X,\mu;E)$ with 
                \[S_{f,\varphi}(g)(x):=\langle g(\varphi x), f(x)\rangle g(\varphi x)\]
          for given $f\in L^\infty(X;E^\prime)$ is paCb.
              \item\label{BadForPolynomialsExample} Let $\alpha \in \mathbb{R}\setminus\mathbb{Q}$. The operator $U_\alpha:L^1([0,1],m)\rightarrow L^1([0,1],m)$ given by $(U_\alpha f)(x) = e(x)f(x+\alpha)$ is a strongly mixing $L^1-L^\infty$ contraction that is also a spaCb operator. In light of the previous examples, it suffices to show that $U$ is strongly mixing. To this end, let $\epsilon > 0$ be arbitrary, let $g(x) = \sum_{j = -k}^kc_je(jx)$ be such that $\|f-g\|_1 < \epsilon$, and observe that for any $h \in L^\infty([0,1],m)$ with $\|h\|_\infty \le 1$ we have

        \begin{alignat*}{2}
            &\limsup_{N\rightarrow\infty}\left|\langle U^nf,h\rangle\right| \le \epsilon+\limsup_{N\rightarrow\infty}\left|\langle U^ng,h\rangle\right| \le \epsilon+\sum_{j = -k}^k|c_j|\limsup_{N\rightarrow\infty}\left|\int_0^1 e((j+N)x)\overline{h(x)}dx\right| = \epsilon,
        \end{alignat*}
        where the final equality follows from the Riemann-Lebesgue lemma.

        \item\label{MixingDoesNotImplypaCb} We now give an example of a $L^1-L^\infty$ contraction that is strongly mixing but not paCb in order to show that the paCb property is not implied by a mixing property. As in the introduction, let $M_e:L^1([0,1],m)\rightarrow L^1([0,1],m)$ be given by $(M_ef)(x) = e(x)f(x)$. We have already seen $M_e$ is a $L^1-L^\infty$ contraction that is not paCb, so it only remains to show that $M_e$ is strongly mixing. To see this, we observe that for any $f \in L^1([0,1],\mu)$ and $h \in L^\infty([0,1],\mu)$ we have

        \begin{equation}
            \lim_{N\rightarrow\infty}\langle M_e^Nf,h\rangle = \lim_{N\rightarrow\infty}\int_0^1e(Nx)f(x)\overline{h}(x)d\mu(x) = 0,
        \end{equation}
        where the final equality follows from the Riemann-Lebesgue lemma.

%        \item Let $V:L^1([0,1],m;E)\rightarrow L^1([0,1],m;E)$, $f\mapsto Vf$ such as 
 %       \[(Vf)(s):=\int_0^s f(t) dt.\]
  %      Then 
%        \[\langle Vf,g\rangle = \int_0^1 \langle Vf(s),g(s)\rangle ds=\int_0^1  \int_0^s\langle f(t),g(s)\rangle dtds=\int_0^1 \langle f(t),\int_t^1g(s)ds\rangle
 %       dt\]
  %      holds for all $g\in L^\infty([0,1],m,E^\prime)$. Hence $V$ is a spaCb $L^1-L^\infty$ contraction. 
  
  \item\label{paCbButNotME} The orbit of any $f\in L^\infty(X,\mu)$ under a spaCb operator is relatively compact in the weak topology of $L^1(X,\mu)$, since the $L^\infty$-unit ball is weakly compact in $L^1(X,\mu)$. Hence every spaCb operator is mean ergodic (see e.g. \cite{OTAoET}, Theorem 8.20). We will now give an example of linear contractions $T,S:L^1([0,1],m)\rightarrow L^1([0,1],m)$ that are not mean ergodic even though they are both paCb. In particular, $T$ and $S$ are not spaCb. In fact, $T$ is not a bounded operator on $L^\infty([0,1],m)$ and $\|S^m\|_\infty = 2^m$. Let $T$ and $S$ be given by

        \begin{alignat*}{2}
            &Tf = \sum_{n = 0}^\infty 2^{(n+1)^2+1}\left(\int_{2^{-n-1}}^{2^{-n}}f(y)dy\right)\mathbbm{1}_{[2^{-(n+1)^2-1},2^{-(n+1)^2}]}\text{, and}\\
            &Sf = \sum_{n = 0}^\infty 2^{n+2}\left(\int_{2^{-n-1}}^{2^{-n}}f(y)dy\right)\mathbbm{1}_{[2^{-n-2},2^{-n-1})}.
        \end{alignat*}
        Since $\|T\mathbbm{1}_{[2^{-n-1},2^{-n}]}\|_\infty = 2^{(n+1)^2-n}$, we see that $T$ does not map $L^\infty([0,1],m)$ into itself. To see that $\|S^m\|_\infty = 2^m$, we see that for any $f$ with $\|f\|_\infty = M < \infty$ and any $n \ge 0$, we have 

        \begin{alignat*}{2}
            &\|(S^mf)\mathbbm{1}_{[2^{-n-1-m},2^{-n-m})}\|_\infty = 2^{n+1+m}\left|\int_{2^{-n-1}}^{2^{-n}}f(y)dy\right| \le 2^{n+1+m}\int_{2^{-n-1}}^{2^{-n}}Mdy = 2^mM\text{, and}\\
            &\|(S^mf)\mathbbm{1}_{[2^{-k-1},2^{-k})}\|_\infty = 0\text{ for }0 \le k < m.
        \end{alignat*}
        
        To see that $T$ is a contraction on $L^1$, we see that for $f \in L^1$ we have

        \begin{alignat*}{2}
            &\|Tf\|_1 = \int_0^1|(Tf)(x)|dx \le \sum_{n = 0}^\infty \int_0^12^{(n+1)^2+1}\left|\int_{2^{-n-1}}^{2^{-n}}f(y)dy\right|\mathbbm{1}_{[2^{-(n+1)^2-1},2^{-(n+1)^2}]}(x)dx\\
            =&\sum_{n = 0}^\infty\left|\int_{2^{-n-1}}^{2^{-n}}f(y)dy\right| \le \sum_{n = 0}^\infty\int_{2^{-n-1}}^{2^{-n}}|f(y)|dy = \int_0^1|f(y)|dy = \|f\|_1.
        \end{alignat*}
        A similar calculation shows that $S$ is a contraction on $L^1$. To see that $R \in \{T,S\}$ is paCb, it suffices to observe that for any $f \in L^1([0,1],m)$ and $x \in (0,1]$, we have that $R^nf(x) = 0$ for all $n \ge \lceil-\log_2(x)\rceil+1$. It remains to show that $T$ and $S$ is are not mean ergodic. Let $R \in \{T,S\}$, let $a_1 = 1$, and let $a_n = (a_{n-1}+1)^2$ if $R = T$, and let $a_n = a_{n-1}+1$ if $R = S$. Let $g = \sum_{n = 1}^\infty \mathbbm{1}_{[2^{-a_{b_n}-1},2^{-a_{b_n}}]}$, where $(b_n)_{n = 1}^\infty$ is an enumeration of $\bigcup_{m = 1}^\infty[2^{2m},2^{2m+1})$, and let $f = \mathbbm{1}_{[\frac{1}{2},1]}$. Since $R^nf = 2^{a_n}\mathbbm{1}_{[2^{-a_n-1},2^{-a_n}]}$, we see that

        \begin{alignat*}{2}
            &\limsup_{N\rightarrow\infty}\frac{1}{N}\sum_{n = 1}^N\langle R^nf,g\rangle \ge \lim_{m\rightarrow\infty}\frac{1}{2^{2m+1}}\sum_{j = 1}^m2^{2j-1} = \lim_{m\rightarrow\infty}\frac{4^{m+1}-1}{6\cdot2^{2m+1}} = \frac{1}{3},\text{ while}\\
            &\liminf_{N\rightarrow\infty}\frac{1}{N}\sum_{n = 1}^N\langle R^nf,g\rangle \le \lim_{m\rightarrow\infty}\frac{1}{2^{2m+2}}\sum_{j = 1}^m2^{2j-1} = \frac{1}{6}.
        \end{alignat*}

        \item We recall that for $1 \le p \le \infty$, $\ell^p(\mathbb{Z}) = L^p(\mathbb{Z},\nu)$ where $\nu$ is the counting measure on $\mathbb{Z}$. For the left shift $\varphi:\mathbb{Z}\rightarrow\mathbb{Z}$ given by $\varphi(n) = n+1$, the Koopman operator $T_\varphi$ is not mean ergodic on $\ell^1(\mathbb{Z})$ even though it is a $\ell^1(\mathbb{Z})-\ell^\infty(\mathbb{Z})$ contraction. We can also verify that $T_\varphi$ is paCb (hence spaCb) since for any $f \in \ell^1(\mathbb{Z})$ we have

        \begin{equation}
            \lim_{n\rightarrow\infty}f(n) = 0\text{, hence }\lim_{N\rightarrow\infty}\frac{1}{N}\sum_{n = 1}^N(T_\varphi^nf)(m) = 0,
        \end{equation}
        for all $m \in \mathbb{Z}$. In particular, a $L^1-L^\infty$ contraction that is also spaCb need not be mean ergodic. We now give another example of this phenomenon on a probability space.
        
        Let $\alpha \in \mathbb{R}\setminus\mathbb{Q}$ be arbitrary and let $S:L^1([0,1],m;\ell^1(\mathbb{Z}))\rightarrow L^1([0,1],m;\ell^1(\mathbb{Z}))$ be given by $(Sf)(x) = T_\varphi(f(x+\alpha))$, and note that $S$ is a spaCb operator as a result of Examples \ref{ExamplesOfOperators}\eqref{ErgodicComposedWithMultiplicationIsspaCb}. However, if $\xi \in \ell^1(\mathbb{Z})$ is such that

        \begin{equation}
            \lim_{N\rightarrow\infty}\frac{1}{N}\sum_{n = 1}^NT_\varphi^n\xi,
        \end{equation}
        does not exist in the norm topology of $\ell^1(\mathbb{Z})$, then for $\xi\mathbbm{1}_X \in L^1(X,\mu;\ell^1(\mathbb{Z}))$, we see that

        \begin{equation}
            \lim_{N\rightarrow\infty}\frac{1}{N}\sum_{n = 1}^NS^n\xi\mathbbm{1}_X,
        \end{equation}
        does not converge in the norm topology of $L^1(X,\mu;\ell^1(\mathbb{Z}))$.

        \item\label{NoncontractiveExample} We now consider an example of an operator $S$ that is spaCb, not a positive operator, and satisfies $\|S\|_1 = \|S\|_{\infty} = 2$. Let $\alpha \in (0,\frac{1}{2})\setminus\mathbb{Q}$ and let $\varphi:[0,1]\rightarrow[0,1]$ be given by $\varphi(x) = x+\alpha\pmod{1}$. Let $F = 2i\mathbbm{1}_{[0,\alpha)}+\frac{1}{2i}\mathbbm{1}_{[\alpha,2\alpha)}+\mathbbm{1}_{[2\alpha,1]}$, and let $S = M_FT_\varphi$. We see that $S^n = M_{F_n}T_\varphi^n$, for some $F_n:[0,1]\rightarrow\{\frac{1}{2i},1,2i\}$, so $\|S^n\|_1 = \|S^n\|_\infty = 2$ for all $n \in \mathbb{N}$. To see that $S$ is spaCb, we recall that $\varphi$ is ergodic, hence for any $f \in L^1([0,1],m)$ and a.e. $x \in X$ we have

        \begin{equation}
            \limsup_{N\rightarrow\infty}\frac{1}{N}\sum_{n = 1}^N|S^nf(x)| \le \lim_{N\rightarrow\infty}\frac{1}{N}\sum_{n = 1}^N2|f(\varphi^nx)| = 2\int_0^1f(y)dy.
        \end{equation}
    \end{enumerate}
\end{examples}

\begin{remark}\label{LampertiOperators}    Given a $\sigma$-finite measure space $(Y,\mathscr{A},\nu)$ and some $1 \le p < \infty$, an operator $T:L^p(Y,\nu)\rightarrow L^p(Y,\nu)$ is a \textbf{Lamperti operator} if for any $f,g \in L^p(Y,\nu)$ satisfying $fg = 0$, we have $TfTg = 0$. Lamperti \cite{OnTheIsometriesOfCertainFunctionSpaces} showed that any isometry on $L^p(Y,\nu)$ for $p \in [1,2)\cup(2,\infty)$ is a Lamperti operator. Kan \cite{ErgodicPropertiesOfLampertiOperators} showed that if $T$ is a Lamperti operator then $T = M_fT_\phi$ for some measurable $f:Y\rightarrow\mathbb{C}$ and some measurable $\phi:Y\rightarrow Y$ satisfying $\mu(\phi^{-1}A) = 0$ when $\mu(A) = 0$. 
        Since contractive Lamperti operators on $L^p(Y,\nu)$ with $p \in (1,\infty)$ satisfy the pointwise ergodic theorem (see \cite{ErgodicPropertiesOfLampertiOperators}), they are good candidates for operators satisfying a paCb-like property on $L^p$.
        However, within the set $S$ of positive invertible isometries on $L^1([0,1],m)$, the subset of elements satisfying the pointwise ergodic theorem is a set of first category in the strong operator topology (see \cite{L1PETFailsForMostInvertibleIsometries}), which suggests that they are unlikely to be paCb.
        Furthermore, a Lamperti operator $T = M_fT_\phi$ will not be spaCb if $f$ is unbounded.
\end{remark}

%%%%%%%%%%%%%%%%%%%%%%%%%%%%%%%%%%%%%%%%%%%%%%%%%%%%%%%%%%%%%%%%%%%%%%%%%%%%%%%%%%%%%%%%%%%%%%%%%%%%%%%%%

\subsection{Sequence spaces and Ultraproducts}\label{UltraproductsSubsection}
\subsubsection{Sequence spaces}
We now identify some spaces that arise naturally when considering sequences of the form $(T^nf(x))_{n = 1}^\infty$ with $f \in L^1(X,\mu;E)$. Later, we will consider mappings of these spaces into ultraproduct spaces. 

For a given Banach space $E$ we define
\begin{alignat*}{2}
    &\text{ces}_\infty(E):= \left\{(x_n)_{n = 1}^\infty\in E^\mathbb{N}\ |\ \sup_{N \in \mathbb{N}}\frac{1}{N}\sum_{n = 1}^N\|x_n\| < \infty\right\},\text{ and}\\
    &A(E):= \Bigg\{(x_n)_{n = 1}^\infty\in \text{ces}_\infty(E)\ |\ \forall\ \epsilon > 0\ \exists\ (e_n)_{n = 1}^\infty \subseteq E\text{ with }\sup_{n \in \mathbb{N}}\|e_n\| < \infty\text{ and }\\
    &\textcolor{white}{A(E):= \Big\{(x_n)_{n = 1}^\infty\in \text{ces}_\infty(E)\ |\ }\limsup_{N \rightarrow \infty}\frac{1}{N}\sum_{n = 1}^N\|x_n-e_n\| < \epsilon\Bigg\}
\end{alignat*}
Then $(\text{ces}_\infty(E),\|\cdot\|_c)$ is a normed space with 
\[\|(x_n)_{n = 1}^\infty\|_c:=\sup_{N \in \mathbb{N}}\frac{1}{N}\sum_{n = 1}^N\|x_n\|.\]
It is worth mentioning that the space $\text{ces}_\infty(\mathbb{C})$ is just the Ces\`aro sequence space ces$_\infty$ that is discussed in \cite{CesaroFunctionSpaces} and the references therein. We also observe that the product topology on $E^n$ is induced by the norm 
\begin{equation}
    \|(x_1,\cdots,x_n)\|_n := \frac{1}{n}\sum_{k = 1}^n\|x_n\|.
\end{equation}
\begin{lemma}
The space $(\text{ces}_\infty(E),\|\cdot\|_c)$ is a Banach space.    
\end{lemma}
 \begin{proof}
We only need to show that $(\text{ces}_\infty(E),\|\cdot\|_c)$ is complete. We see that if $((x_{n,m})_{n = 1}^\infty)_{m = 1}^\infty$ is a Cauchy sequence in $\text{ces}_\infty(E)$, then $(x_{1,m},\cdots,x_{n,m})_{m = 1}^\infty$ is a Cauchy sequence in $E^n$ since $\|(y_{1,m},\cdots,y_{n,m})\|_n \le \|(y_{n,m})_{n = 1}^\infty\|_c$, so $((x_{1,m},\cdots,x_{n,m}))_{m = 1}^\infty$ converges in $E^n$ to some $(z_{1,n},\cdots,z_{n,n})$. Since for $n_1 < n_2$, $E^{n_1}$ is naturally identified with a closed subspace of $E^{n_2}$, we see that $z_{i,n_1} = z_{i,n_2} = z_i$ for $1 \le i \le n_1$. Now let us assume for the sake of contradiction that $((x_{n,m})_{n = 1}^\infty)_{m = 1}^\infty$ does not converge to $(z_n)_{n = 1}^\infty$ in $\|\cdot\|_c$ as $m\rightarrow\infty$. By passing to a subsequence $((x_{n,m_k})_{n = 1}^\infty)_{k = 1}^\infty$, which we will again denote by $((x_{n,m})_{n = 1}^\infty)_{m = 1}^\infty$, we may assume without loss of generality that for some $\epsilon > 0$ we have $\|(x_{n,m}-z_n)_{n = 1}^\infty\|_c > \epsilon$ for all $m\in\mathbb{N}$, while $\lim_{m\rightarrow\infty}x_{n,m} = z_n$ for all $n \in \mathbb{N}$. We inductively construct $(N_k)_{k = 1}^\infty$ and $(m_k)_{k = 1}^\infty$ for which $\|x_{n,m_k}-z_n\| < \frac{\epsilon}{2}$ for all $1 \le n \le N_k$, and $\frac{1}{N_{k+1}}\sum_{n = 1}^{N_{k+1}}\|x_{n,m_k}-z_n\| > \epsilon$. We now see that for any $k \in \mathbb{N}$ we have

\begin{alignat*}{2}
    &\|(x_{n,m_k}-x_{n,m_{k+1}})_{n = 1}^\infty\|_c \ge\frac{1}{N_{k+1}}\sum_{n = 1}^{N_{k+1}}\|x_{n,m_k}-x_{n,m_{k+1}}\|\\
    \ge&\frac{1}{N_{k+1}}\sum_{n = 1}^{N_{k+1}}\|x_{n,m_k}-z_n\|-\frac{1}{N_{k+1}}\sum_{n = 1}^{N_{k+1}}\|z_n-x_{n,m_{k+1}}\| > \frac{\epsilon}{2},
\end{alignat*}
which contradicts the fact that $((x_{n,m_k})_{n = 1}^\infty)_{k = 1}^\infty$ is a Cauchy sequence.
\end{proof}
To see that $A(E)$ is a closed subspace of $\text{ces}_\infty(E)$, it suffices to observe that

\begin{alignat*}{2}
    &A(E)= \Bigg\{(x_n)_{n = 1}^\infty\in \text{ces}_\infty(E)\ |\ \forall\ \epsilon > 0\ \exists\ (e_n)_{n = 1}^\infty \subseteq E\text{ with }\sup_{n \in \mathbb{N}}\|e_n\| < \infty\text{ and }\\
    &\textcolor{white}{A(E):= \Big\{(x_n)_{n = 1}^\infty\in \text{ces}_\infty(E)\ |\ }\sup_{N \in \mathbb{N}}\frac{1}{N}\sum_{n = 1}^N\|x_n-e_n\| < \epsilon\Bigg\}.
\end{alignat*}
%%%%%%%%%%%%%%%%%%%%%%%%%%%%%%%%%%%%%%%%%%%%%%%%%%%%%%%%%%%%%%%%%
\subsubsection{Ultraproducts}
The goal of this subsection is to define the space $A_\mathcal{U}(E)$, which is a closed subspace of an ultraproduct of Banach spaces. For a more detailed discussion of the usage of ultraproducts in Banach space theory, the reader is referred to \cite{UltraproductsInBanachSpaceTheory}.\\

\noindent Let $(E_n,\|\cdot\|_n)_{n = 1}^\infty$ be a family of Banach spaces and $\mathcal{U}\in \beta\mathbb{N}^*$. Consider the Banach space
\[\ell^\infty(E_n)=\{(x_n)_{n = 1}^\infty\in\prod_{n = 1}^\infty E_n:\sup_{n \in \mathbb{N}}\|x_n\|_n<\infty\}\]
and its closed subspace 
\[N_{\mathcal{U}}=\{(x_n)_{n = 1}^\infty\in \ell^\infty(E_n): \mathcal{U}-\lim_{n\rightarrow\infty}\|x_n\|_n =0\}.\]
We call the quotient space 
\[(E_n)_{\mathcal{U}}:=\ell^\infty(E_n)/N_{\mathcal{U}}\]
the \textbf{ultraproduct} of $(E_n)_{n = 1}^\infty$ (with respect to $\mathcal{U}$). We denote its elements by $[(x_n)_{n = 1}^\infty]_{\mathcal{U}}^*$.
The norm on $(E_n)_{\mathcal{U}}$ is given by 
\[\|[(x_n)_{n = 1}^\infty]_\mathcal{U}^*\|^*_{\mathcal{U}}=\mathcal{U}-\lim_n\|x_n\|_n.\]
Now let us fix a Banach space $E$ and for $n\in\mathbb{N}$ let $E_n:=E^n$ be given the norm 
\begin{equation}
    \|(x_1,\cdots,x_n)\|_n := \frac{1}{n}\sum_{k = 1}^n\|x_k\|.
\end{equation}
We define 
\[i: \text{ces}_\infty(E)\rightarrow \ell^\infty (E_n),\  (x_n)_{n = 1}^\infty\mapsto (x_n(k)_{k\leq n})_{n = 1}^\infty\]
with $x_n(k)=x_k$ for all $k,n\in\mathbb{N}$. Since $i$ is a linear isometry, $i(A(E))$ is a closed subspace of $\ell^\infty (E_n)$. We consider the space 
\[A_{\mathcal{U}}(E):=i(A(E))/(i(A(E))\cap N_\mathcal{U}).\] 
To see that $A_\mathcal{U}(E)$ is naturally identified as a closed subspace of $(E_n)_\mathcal{U}$, it suffices to observe that $i((x_n)_{n = 1}^\infty)+(i(A(E))\cap N_\mathcal{U})\mapsto i((x_n)_{n = 1}^\infty)+N_\mathcal{U}$ is a well defined linear injective contraction since $i(A(E))\cap N_\mathcal{U} \subseteq N_\mathcal{U}$. The norm $\|\cdot\|_\mathcal{U}$ on $A_\mathcal{U}(E)$ is given by
\[\|[i(x_n)_{n = 1}^\infty]_\mathcal{U}\|_\mathcal{U}=\mathcal{U}-\lim_n \frac{1}{n}\sum_{k = 1}^n\|x_n(k)\|=\mathcal{U}-\lim_n \frac{1}{n}\sum_{k = 1}^n\|x_k\|\]
for all $(x_n)_{n=1}^\infty \in A(E)$. 
For many calculations we will write $x_\mathcal{U}$ to denote $[i(x_n)_{n = 1}^\infty]_\mathcal{U}$.
In other situations in which more detail is required, we will write $[(x_n)_{n = 1}^\infty]_\mathcal{U}$ to denote $[i(x_n)_{n = 1}^\infty]_\mathcal{U}$.
\subsubsection{The shift operator}
The shift operator $S$ is given by 
\[S:A(E)\rightarrow A(E), \ (x_n)_{n = 1}^\infty\mapsto (x_{n+1})_{n = 1}^\infty,\]
and we denote by $S_\mathcal{U}$ the induced map 
\[S_\mathcal{U}:A_\mathcal{U}(E)\rightarrow A_\mathcal{U}(E), \ [i(x_n)_{n = 1}^\infty]_\mathcal{U}\mapsto [i(x_{n+1})_{n = 1}^\infty]_\mathcal{U}.\]
on $A_\mathcal{U}(E)$. We first need to show that $S_\mathcal{U}$ is well defined, i.e., $S$ leaves $i(A(E))\cap N_\mathcal{U}$ invariant. We begin by showing that for $(x_n)_{n = 1}^\infty \in A(E)$ we have $\lim_{n\rightarrow\infty}n^{-1}\|x_n\| = 0$. Let $\epsilon > 0$ be arbitrary, and let $(e_n)_{n = 1}^\infty$ be a bounded sequence in $E$ for which $\|(x_n)_{n = 1}^\infty-(e_n)_{n = 1}^\infty\|_c < \epsilon$. We see that

\begin{equation*}
    \epsilon \ge \limsup_{N\rightarrow\infty}\frac{1}{N}\sum_{n = 1}^N\|x_n-e_n\| \ge \limsup_{N\rightarrow\infty}\frac{\|x_N-e_N\|}{N} = \limsup_{N\rightarrow\infty}\frac{\|x_N\|}{N}.
\end{equation*} 
Returning to the well definedness of $S_\mathcal{U}$, let $(x_{n})_{n = 1}^\infty \in A(E)$ be such that 
\[\mathcal{U}-\lim_{N\rightarrow\infty} \frac{1}{N}\sum_{n = 1}^N\|x_n\|=0.\]
We see that
\[\mathcal{U}-\lim_{N\rightarrow\infty} \frac{1}{N}\sum_{n = 1}^N\|x_{n+1}\|=\mathcal{U}-\lim_{N\rightarrow\infty} \frac{1}{N}\sum_{n = 1}^N(\|x_{n+1}\|-\|x_{n}\|)= \mathcal{U}-\lim_{N\rightarrow\infty}\frac{\|x_{N+1}\|-\|x_1\|}{N} = 0.\]
In fact, we even see that $S_\mathcal{U}$ is a linear isometry on $A_\mathcal{U}(E)$. 
We remark that the shift map would not be well defined on $i(\text{ces}_\infty(E))/(i(\text{ces}_\infty(E))\cap N_\mathcal{U})$, and this is the reason that our main results need the spaCb property instead of just the paCb property (see also Lemma \ref{LemmaUsingNiceness}).

\subsubsection{Mixing sequences and pointwise orbits}
\begin{definition}\label{DefinitionOfMixingSequences}
    Let $E$ be a Banach space. We call a sequence $(x_n)_{n = 1}^\infty \in A(E)$
    \begin{enumerate}[(i)]    
        \item \textbf{fully ergodic} if for any $\mathcal{U} \in \beta\mathbb{N}^*$,
        \[\lim_{H\rightarrow\infty}\frac{1}{H}\sum_{h=1}^H \langle S^h_{\mathcal{U}}x_{\mathcal{U}},g^\prime\rangle_\mathcal{U}=0\]
        holds for all $g^\prime\in A_\mathcal{U}(E)^\prime$. 
        \item \textbf{almost weakly mixing} if for any $\mathcal{U} \in \beta\mathbb{N}^*$,
        \[\lim_{H\rightarrow\infty}\frac{1}{H}\sum_{h=1}^H |\langle S^h_{\mathcal{U}}x_{\mathcal{U}},g^\prime\rangle_\mathcal{U}|=0\]
        holds for all $g^\prime\in A_\mathcal{U}(E)^\prime$. 
        \item \textbf{almost mildly mixing} if for any $\mathcal{U} \in \beta\mathbb{N}^*$, 
               \[ \mathcal{IP}^*-\lim_{h\rightarrow\infty}\langle S^h_{\mathcal{U}}x_{\mathcal{U}},g^\prime\rangle_\mathcal{U} = 0.\]
        holds for all $g^\prime\in A_\mathcal{U}(E)^\prime$.
        \item \textbf{almost strongly mixing} if for any $\mathcal{U} \in \beta\mathbb{N}^*$, \[ \lim_{h\rightarrow\infty}\langle S^h_{\mathcal{U}}x_{\mathcal{U}},g^\prime\rangle_\mathcal{U} = 0.\]
holds for all $g^\prime\in A_\mathcal{U}(E)^\prime$.
    \end{enumerate}
\end{definition}

\begin{lemma}\label{LemmaUsingNiceness}
    Let $E$ be a Banach space, let $(X,\mathscr{B},\mu)$ be a $\sigma$-finite measure space, and let $T:L^1(X,\mu;E)\rightarrow L^1(X,\mu;E)$ be a bounded linear paCb operator. If $f \in L^1(X,\mu;E)$, then for a.e. $x \in X$ we have $(T^nf(x))_{n = 1}^\infty \in \text{ces}_\infty(E)$. Furthermore, if $T$ is spaCb, then for a.e. $x \in X$ we have $(T^nf(x))_{n = 1}^\infty \in A(E)$. 
\end{lemma}

\begin{proof}
    The first claim follows immediately from the definition of paCb. To see the second claim, let $\epsilon > 0$ be arbitrary and let $M_\epsilon \in \mathbb{N}$ be such that for $f_\epsilon := f\mathbbm{1}_{f < M_\epsilon} \in L^\infty(X,\mu;E)\cap L^1(X,\mu;E)$ we have $\|f-f_\epsilon\|_1 < \epsilon$. Since $T$ is paCb, we observe for a.e. $x \in X$ that
    \begin{equation*}
        \limsup_{N\rightarrow\infty}\frac{1}{N}\sum_{n = 1}^N\|T^nf(x)-T^nf_\epsilon(x)\| = \limsup_{N\rightarrow\infty}\frac{1}{N}\sum_{n = 1}^N\|T^n(f-f_{\epsilon})(x)\| \le C\int_X\|f-f_{\epsilon}\|d\mu < C\epsilon.
    \end{equation*}
    Since $T$ is spaCb, we see for a.e. $x\in X$ that $(T^nf_{\epsilon}(x))_{n = 1}^\infty$ is a bounded sequence.
\end{proof}
%%%%%%%%%%%%%%%%%%%%%%%%%%%%%%%%%%%%%%%%%%%%%%%%%%%%%%%%%%%%%%%%%%%%%%%%%%%%%%%%%%%%%%%%%%%%%%%%%%%%%%%%%%%%%%%%%%%%%%%%

\subsection{Polynomial pointwise theorems}

We begin by recalling the uniform polynomial Wiener-Wintner Theorem of Lesigne.

\begin{theorem}[{Lesigne, \cite{PolynomialWW}}]\label{LesignePolynomialTheorem}
    Let $(X,\mathscr{B},\mu,\varphi)$ be a weakly mixing measure preserving system and let $\mathbb{R}_k[x]$ denote the collection of polynomials in $\mathbb{R}[x]$ of degree at most $k$. For any $f \in L^1(X,\mu)$ with $\int_Xfd\mu = 0$, and continuous $\phi:\mathbb{T}\rightarrow\mathbb{R}$, and for a.e. $x \in X$ we have

    \begin{equation}
        \lim_{N\rightarrow\infty}\sup_{p \in \mathbb{R}_k[x]}\left|\frac{1}{N}\sum_{n = 1}^Nf(\varphi^nx)\phi(p(nx))\right| = 0.
    \end{equation}
\end{theorem}
We now given an example showing that pointwise ergodic theorems involving polynomial weights require more than just the mixing properties and the spaCb property. Let $\alpha \in \mathbb{R}\setminus\mathbb{Q}$ be arbitrary and recall that in Example \ref{ExamplesOfOperators}\eqref{BadForPolynomialsExample} we showed the operator $U_\alpha:L^1([0,1],m)\rightarrow L^1([0,1],m)$ given by $(U_\alpha f)(x) = e(x)f(x+\alpha)$ is a strongly mixing spaCb operator. We see that for $f(x) = e(x)$, we have 

\begin{equation}
    U_\alpha^nf(x) = e\left(\binom{n+1}{2}\alpha\right)e((n+1)x).
\end{equation}
For each $x \in [0,1]$, let $p_x(y) \in \mathbb{R}[y]$ be given by

\begin{equation}
    p_x(y) = -\frac{(y+1)y}{2}\alpha-(y+1)x\text{ and note that }\lim_{N\rightarrow\infty}\frac{1}{N}\sum_{n = 1}^NU_\alpha^nf(x)e(p_x(n)) = 1.
\end{equation}

The key property possessed by Koopman operators that is not possessed by $U_\alpha$, is that Koopman operators are an algebra automorphisms of $L^\infty(X,\mu)$. If $E$ is a Banach algebra, let $I(E)$ denote the collection of algebra automorhpism of $E$ that are also isometries. We see that if $F:X\rightarrow I(E)$ is strongly measurable with respect to the Borel $\sigma$-algebra of the strong operator topology and $\varphi:X\rightarrow X$ is measure preserving, then the operator $S$ given by $(Sf)(x) = F(x)(f(\varphi x))$ is an algebra automorphism of $L^\infty(X,\mu;E)$. It is natural to ask whether or not there is an analogue of Lesigne's Theorem for weakly mixing operators having the form of $S$, or if there is an analogue in the non-commutative setting (cf. \cite{NoncommutativeMultiParameterWWTheorem}), but we do not pursue this here.

%%%%%%%%%%%%%%%%%%%%%%%%%%%%%%%%%%%%%%%%%%%%%%%%%%%%%%%%%%%%%%%%%%%%%%%%%%%%%%%%%%%%%%%%%%%%%%%%%%%%%%%%%%%%%%%%%%%%%%%%
\section{Main results}
The following lemma establishes the connection between our given operator $T$ on $L^1(X,\mu;E)$ and the shift operator $S_\mathcal{U}$ on $A_\mathcal{U}(E)$.

\begin{lemma}\label{keylemma}
    Let $E$ be a Banach space, let $T:L^1(X,\mu;E)\rightarrow L^1(X,\mu;E)$ be a bounded linear spaCb operator, and let $f \in L^1(X,\mu,E)$. There exists a full measure set $X_f \subseteq X$ such that for any $x \in X_f$, any $\mathcal{U} \in \beta\mathbb{N}^*$, and any $G^\prime \in A_\mathcal{U}(E)^\prime$, there exists $g^\prime \in L^1(X,\mu;E)^\prime$ for which

    \begin{equation}
        \langle T^hf, g^\prime\rangle_{L^1_E} = \langle [(T^{n+h}f(x))_{n = 1}^\infty]_\mathcal{U},G^\prime\rangle_\mathcal{U} = \langle S_\mathcal{U}^h[(T^nf(x))_{n = 1}^\infty]_\mathcal{U},G^\prime\rangle_\mathcal{U}\text{ for all }h \ge 0.
    \end{equation}
\end{lemma}
\begin{proof}
We look at the map
\[R(T, x) : L^1(X, \mu; E) \rightarrow A(E), u \mapsto (T^n u(x))_{n\in\mathbb{N}}.\]
Since $T$ is spaCb, Lemma \ref{LemmaUsingNiceness} shows that $R(T, x)$ is for a.e.\ $x \in X$ a well-defined, bounded linear operator between Banach spaces. For a given $G^\prime \in A_\mathcal{U}(E)^\prime$ we define
\[g_{G^\prime,T,x}(
u
) := (G^\prime \circ q \circ i \circ R(T,x))(u),\]
where
\[i : A(E) \rightarrow \ell^\infty(E^n), (x_n)_{n\in \mathbb{N}} \mapsto ((x_n(k))_{k = 1}^n)_{n\in \mathbb{N}},\]
with $x_n(k) := x_k$ and $q$ is the canonical projection on the quotient space $i(A(E))/(i(A(E))\cap
N_{\mathcal{U}})$. Since the image of $R(T, x)$ is in $A(E)$ the map $q \circ i \circ R(T, x)$ is for a.e.\ $x \in X$ a well-defined map between $L^1(X, \mu; E)$ and $A_{\mathcal{U}}(E)$. Since the considered maps are linear and
continuous for a.e.\ $x \in X$, it suffices to take $g^\prime = g_{G^\prime,T,x}$.
\end{proof}

\begin{theorem}\label{firstapplicationkeylemma}
    Let $T$ be bounded linear spaCb operator on $L^1(X,\mu;E)$ and let $f \in L^1(X,\mu;E)$. %satisfy $\int_Xfd\mu = 0$.
    \begin{enumerate}[(i)]
        \item If $f$ is ergodic with respect to $T$, then the sequence $(T^nf(x))_{n = 1}^\infty$ is fully ergodic for almost every $x\in X$.

        \item If $f$ is weakly mixing with respect to $T$, then $(T^nf(x))_{n = 1}^\infty$ is almost weakly mixing for almost every $x\in X$.

        \item If $f$ is mildly mixing with respect to $T$, then $(T^nf(x))_{n = 1}^\infty$ is almost mildly mixing for almost every $x\in X$.

        \item If $f$ is strongly mixing with respect to $T$, then $(T^nf(x))_{n = 1}^\infty$ is almost strongly mixing for almost every $x\in X$.
    \end{enumerate}
\end{theorem}

\begin{proof}
Since $T$ is spaCb, Lemma \ref{LemmaUsingNiceness} tells us that $(T^nf(x))_{n = 1}^\infty \in A(E)$. Let $\mathcal{U} \in \beta\mathbb{N}^*$ and  $G^\prime \in A_{\mathcal{U}}(E)^\prime$. Since $T$ is spaCb, by Lemma \ref{keylemma} there exists $g^\prime \in L^1(X,\mu;E)^\prime$ for which
    \begin{equation}
        \langle T^hf, g^\prime\rangle_{L^1_E} = \langle S_\mathcal{U}^h[(T^nf(x))_{n = 1}^\infty]_\mathcal{U},G^\prime\rangle_{\mathcal{U}}
    \end{equation}     
    holds for all $h\in\mathbb{N}$. Since $f$ is ergodic 
\begin{alignat*}{2}
0&=\lim_{H\rightarrow \infty}\frac{1}{H}\sum_{h=1}^H\langle T^h f,g^\prime\rangle_{L_E^1} = \lim_{H\rightarrow \infty}\frac{1}{H}\sum_{h=1}^H\langle S_\mathcal{U}^h[(T^nf(x))_{n = 1}^\infty]_\mathcal{U},G^\prime\rangle_\mathcal{U}
\end{alignat*}
follows, hence $(T^nf(x))_{n = 1}^\infty$ is a fully ergodic sequence. Statements (ii)-(iv) follow similarly.  \end{proof}
\begin{corollary}\label{corollarmixingkoopman}
Let $(X,\mathscr{B},\mu,\varphi)$ be a measure preserving system, $T_\varphi$ be the Koopman operator on $L^1(X,\mu;E)$ and $f\in L^1(X,\mu;E)$ with $\int_X f d\mu =0$. If $(X,\mathscr{B},\mu,\varphi)$ is ergodic / weakly mixing / mildly mixing / strongly mixing, then for almost every $x\in X$ the sequence $(T_\varphi^nf(x))_{n = 1}^\infty$ is fully ergodic /  almost weakly mixing /  almost mildly mixing /  almost strongly mixing.
    
\end{corollary}
    \noindent Corollary \ref{corollarmixingkoopman} follows directly from Theorem \ref{meanergodic} and Theorem \ref{firstapplicationkeylemma}. 

\begin{theorem}\label{UniformConvergenceForCESequences}
    Let $E$ be a Banach space, let $(\delta_n)_{n = 1}^\infty \subseteq \mathbb{R}^+$ satisfy $\lim_{n\rightarrow\infty}\delta_n = 0$, and let $\vec{e} = (e_n)_{n = 1}^\infty \in A(E)$ be fully ergodic. We have that

    \begin{alignat*}{2}
        &\lim_{N\rightarrow\infty}\sup_{(c_n)_{n = 1}^N \in \mathcal{I}(N,\delta_N)}\left|\left|\frac{1}{N}\sum_{n = 1}^Ne_nc_n\right|\right| = 0\text{, where }\\
        &\mathcal{I}(N,\delta_N) := \left\{(c_n)_{n = 1}^N \in D_1^N\ |\ \frac{1}{N}\sum_{n = 1}^{N-1}|c_n-c_{n+1}| < \delta_N\right\}.
    \end{alignat*}
\end{theorem}

\begin{proof}
    We proceed by way of contradiction. Let us assume that there is some $\epsilon > 0$, and sequences $(N_m)_{m = 1}^\infty \subseteq \mathbb{N}$ and $((c_{n,m})_{n = 1}^{N_m})_{m = 1}^\infty$ for which

    \begin{equation}
        \xi_m := \frac{1}{N_m}\sum_{n = 1}^{N_m}e_nc_{n,m}\text{ satisfies }\|\xi_m\| \ge \epsilon.
    \end{equation}
    Let $\mathcal{U} \in \beta\mathbb{N}^*$ be such that $\{N_m\ |\ m\in\mathbb{N}\} \in \mathcal{U}$, and let $\mathcal{V} \in \beta\mathbb{N}^*$ be as in Lemma \ref{PushforwardOfUltrafilters}. Let $(f_m^\prime)_{m = 1}^\infty \subseteq E^\prime$ be such that $\|f_m^\prime\| = 1$ and $f_m^\prime(\xi_m) = \|\xi_m\|$. Let $g_N^\prime = f_m^\prime$ and $\gamma_{n,N} = c_{n,m}$ for $N_{m-1} < N \le N_m$. Let $G^\prime \in A_\mathcal{U}(E)^\prime$ be given by
    
    \begin{equation}
        G^\prime(x_\mathcal{U}) = \mathcal{U}-\lim_{N\rightarrow\infty}\frac{1}{N}\sum_{n = 1}^Ng_N^\prime(x_n\gamma_{n,N}) = \mathcal{V}-\lim_{m\rightarrow\infty}\frac{1}{N_m}\sum_{n = 1}^{N_m}f_m^\prime(x_nc_{n,m}).
    \end{equation}
    For $M \in \mathbb{N}$ let $e_n(M) = \frac{e_n}{\|e_n\|}\text{min}(\|e_n\|,M)$ and observe that

    \begin{alignat*}{2}
        &\left|G^\prime(S_{\mathcal{U}}e(M)_\mathcal{U})-G^\prime(e(M)_\mathcal{U})\right| = \left|\mathcal{V}-\lim_{m\rightarrow\infty}\frac{1}{N_m}\sum_{n = 1}^{N_m}f_m^\prime((e_{n+1}(M)-e_n(M))c_{n,m})\right|\\
        =& \left|\mathcal{V}-\lim_{m\rightarrow\infty}\frac{1}{N_m}\sum_{n = 2}^{N_m+1}f_m^\prime(e_n(M)(c_{n-1,m}-c_{n,m}))\right| \le \mathcal{V}-\lim_{m\rightarrow\infty}\frac{1}{N_m}\sum_{n = 2}^{N_m+1}M|c_{n-1}-c_n| = 0,\text{ hence}\\
        &G^\prime(S_{\mathcal{U}}e_\mathcal{U}) = \lim_{M\rightarrow\infty}G^\prime(S_{\mathcal{U}}e(M)_\mathcal{U}) = \lim_{M\rightarrow\infty}G^\prime(e(M)_\mathcal{U}) = G^\prime(e_\mathcal{U})\text{, thus}\\
        &G^\prime(e_\mathcal{U}) = \lim_{H\rightarrow\infty}\frac{1}{H}\sum_{h = 1}^HG^\prime(S_{\mathcal{U}}^he_\mathcal{U}) = 0\text{, but}\\
        &\left|G^\prime(e_\mathcal{U})\right| = \left|\mathcal{V}-\lim_{m\rightarrow\infty}\frac{1}{N_m}\sum_{n = 1}^{N_m}f_m^\prime(e_nc_{n,m})\right| = \left|\mathcal{V}-\lim_{m\rightarrow\infty}f_m^\prime(\xi_m)\right| = \mathcal{V}-\lim_{m\rightarrow\infty}\|\xi_m\| \ge \epsilon.
    \end{alignat*}
\end{proof}

\begin{theorem}\label{UniformConvergenceForNWMSequences}
    Let $E$ be a Banach space, let $(\delta_n)_{n = 1}^\infty \subseteq \mathbb{R}^+$ satisfy $\lim_{n\rightarrow\infty}\delta_n = 0$, and let $\vec{e} = (e_n)_{n = 1}^\infty \in A(E)$ be almost weakly mixing. We have that

    \begin{alignat*}{2}
        &\lim_{N\rightarrow\infty}\sup_{(c_n)_{n = 1}^N \in \mathcal{C}(N,\delta_N)}\left|\left|\frac{1}{N}\sum_{n = 1}^Ne_nc_n\right|\right| = 0\text{, where }\\
        &\mathcal{C}(N,\delta_N) := \left\{(c_n)_{n = 1}^N \in D_1^N\ |\ \exists\ \lambda \in \mathbb{S}^1\text{ such that }\frac{1}{N}\sum_{n = 1}^{N-1}|\lambda c_n-c_{n+1}| < \delta_N\right\}\\
        &\textcolor{white}{\mathcal{C}(N,\delta_N) :}=\left\{(\lambda^nc_n)_{n = 1}^N\ |\ \lambda \in \mathbb{S}^1\ \&\ (c_n)_{n = 1}^N \in \mathcal{I}(N,\delta_N)\right\}.
    \end{alignat*}
\end{theorem}

\begin{proof}
    We proceed by way of contradiction. Let us assume that there is some $\epsilon > 0$, and sequences $(N_m)_{m = 1}^\infty \subseteq \mathbb{N}$, $(\lambda_m)_{m = 1}^\infty \subseteq \mathbb{S}^1$, and $((c_{n,m})_{n = 1}^{N_m})_{m = 1}^\infty$ for which

    \begin{equation}
        \xi_m := \frac{1}{N_m}\sum_{n = 1}^{N_m}e_n\lambda_m^nc_{n,m}\text{ satisfies }\|\xi_m\| \ge \epsilon.
    \end{equation}
    Let $\mathcal{U} \in \beta\mathbb{N}^*$ be such that $\{N_m\ |\ m\in\mathbb{N}\} \in \mathcal{U}$, and let $\mathcal{V} \in \beta\mathbb{N}^*$ be as in Lemma \ref{PushforwardOfUltrafilters}. Let $(f_m^\prime)_{m = 1}^\infty \subseteq E^\prime$ be such that $\|f_m^\prime\| = 1$ and $f_m^\prime(\xi_m) = \|\xi_m\|$. Let $g_N^\prime = f_m^\prime$, $\lambda_N = \lambda_m$, and $\gamma_{n,N} = c_{n,m}$ for $N_{m-1} < N \le N_m$. Let $G^\prime \in A_\mathcal{U}(E)^\prime$ be given by
    
    \begin{equation}
        G^\prime(x_\mathcal{U}) = \mathcal{U}-\lim_{N\rightarrow\infty}\frac{1}{N}\sum_{n = 1}^Ng_N^\prime(x_n\lambda_N^n\gamma_{n,N}) = \mathcal{V}-\lim_{m\rightarrow\infty}\frac{1}{N_m}\sum_{n = 1}^{N_m}f_m^\prime(x_n\lambda_m^nc_{n,m}).
    \end{equation}
    For $M \in \mathbb{N}$ let $e_n(M) = \frac{e_n}{\|e_n\|}\text{min}(\|e_n\|,M)$, let $\lambda =
    \mathcal{V}-\lim_m\lambda_m^{-1}$, and observe that

    \begin{alignat*}{2}
        &\left|G^\prime(S_{\mathcal{U}}e(M)_\mathcal{U})-\lambda G^\prime(e(M)_\mathcal{U})\right| = \left|\mathcal{V}-\lim_{m\rightarrow\infty}\frac{1}{N_m}\sum_{n = 1}^{N_m}f_m^\prime((e_{n+1}(M)-\lambda e_n(M))\lambda_m^nc_{n,m})\right|\\
        =& \left|\mathcal{V}-\lim_{m\rightarrow\infty}\frac{1}{N_m}\sum_{n = 2}^{N_m+1}f_m^\prime(e_n(M)\lambda_m^{n-1}(c_{n-1,m}-\lambda\lambda_mc_{n,m}))\right|\\
        \le &\mathcal{V}-\lim_{m\rightarrow\infty}\frac{1}{N_m}\sum_{n = 2}^{N_m+1}M|c_{n-1}-\lambda\lambda_mc_n| = 0,\text{ hence}\\
        &G^\prime(S_{\mathcal{U}}e_\mathcal{U}) = \lim_{M\rightarrow\infty}G^\prime(S_{\mathcal{U}}e(M)_\mathcal{U}) = \lim_{M\rightarrow\infty}\lambda G^\prime(e(M)_\mathcal{U}) = \lambda G^\prime(e_\mathcal{U})\text{, thus}\\
        &\left|G^\prime(e_\mathcal{U})\right| = \lim_{H\rightarrow\infty}\frac{1}{H}\sum_{h = 1}^H\left|\lambda^hG^\prime(e_\mathcal{U})\right| = \lim_{H\rightarrow\infty}\frac{1}{H}\sum_{h = 1}^H\left|G^\prime\left(S_{\mathcal{U}}^he_\mathcal{U}\right)\right| = 0\text{, but}\\
        &\left|G^\prime(e_\mathcal{U})\right| = \left|\mathcal{V}-\lim_{m\rightarrow\infty}\frac{1}{N_m}\sum_{n = 1}^{N_m}f_m^\prime(e_n\lambda_m^nc_{n,m})\right| = \left|\mathcal{V}-\lim_{m\rightarrow\infty}f_m^\prime(\xi_m)\right| = \mathcal{V}-\lim_{m\rightarrow\infty}\|\xi_m\| \ge \epsilon.
    \end{alignat*}
\end{proof}

\begin{remark}
    Our next result is a uniform pointwise ergodic theorem for mildly mixing systems. We recall that the notion opposite to mild mixing is rigidity. If $(X,\mathscr{B},\mu,\varphi)$ is a measure preserving system and $\lambda \in \mathbb{S}^1$, then $f \in L^2(X,\mu)$ is \textbf{$\lambda$-rigid} if there exists a sequence $(k_w)_{w = 1}^\infty$ for which $\lim_{w\rightarrow\infty}\|T_\varphi^{k_w}f-\lambda f\|_2 = 0$. It is a classical fact that if $f \in L^2(X,\mu)$ is $\lambda$-rigid and $g \in L^2(X,\mu)$ is mildly mixing, then $\langle f,g\rangle = 0$. Our next theorem is a another manifestation of this fact that is uniform with respect to $\lambda$ and with respect to a countable collection of rigidity sequences $\{(k_{w,i})_{w = 1}^\infty\}_{i = 1}^\infty$.
\end{remark}

\begin{theorem}\label{UniformConvergenceForNMMSequences}
    Let $E$ be a Banach space, let $B = (b_n)_{n = 1}^\infty \subseteq \mathbb{N}$ be non-decreasing, let $K = (\{k_{w,i}\}_{i = 1}^{b_w})_{w = 1}^\infty \subseteq \mathbb{N}$ let $(\delta_n)_{n = 1}^\infty \subseteq \mathbb{R}^+$ satisfy $\sum_{n = 1}^\infty\delta_n < \infty$, and let $(e_n)_{n = 1}^\infty \in A(E)$ be almost mildly mixing. We have that

    \begin{alignat*}{2}
        &\lim_{N\rightarrow\infty}\sup_{\underset{\lambda \in \mathbb{S}^1}{(c_n)_{n = 1}^N \in \mathcal{R}(\lambda)}}\left|\left|\frac{1}{N}\sum_{n = 1}^Ne_nc_n\right|\right| = 0\text{, where }\\
        &\mathcal{R}(\lambda) = \mathcal{R}(\lambda,N,\delta_N,K) \\
        &:= \left\{(c_n)_{n = 1}^N \in D_1^N\ |\ \forall\ 1 \le w \le N, \exists 1 \le i \le b_w\text{ s.t. }\frac{2k_{w,i}}{N}+\frac{1}{N}\sum_{n = 1}^{N-k_{w,i}}|\lambda c_n-c_{n+k_{w,i}}| < \delta_w\right\}.
    \end{alignat*}
\end{theorem}

\begin{proof}
    We proceed by way of contradiction. Let us assume that there is some $\epsilon > 0$, and sequences $(N_m)_{m = 1}^\infty \subseteq \mathbb{N}$, $(\lambda_m)_{m = 1}^\infty \subseteq \mathbb{S}^1$, and $(c_{n,m})_{n = 1}^{N_m} \in \mathcal{R}(\lambda_m,N_m,\delta_{N_m},K)$ for which

    \begin{equation}
        \xi_m := \frac{1}{N_m}\sum_{n = 1}^{N_m}e_nc_{n,m}\text{ satisfies }\|\xi_m\| \ge \epsilon.
    \end{equation}
    Let $\mathcal{U} \in \beta\mathbb{N}^*$ be such that $\{N_m\ |\ m\in\mathbb{N}\} \in \mathcal{U}$, and let $\mathcal{V} \in \beta\mathbb{N}^*$ be as in Lemma \ref{PushforwardOfUltrafilters}. Let $(f_m^\prime)_{m = 1}^\infty \subseteq E^\prime$ be such that $\|f_m^\prime\| = 1$ and $f_m^\prime(\xi_m) = \|\xi_m\|$. Let $g_N^\prime = f_m^\prime$, $\lambda_N = \lambda_m$, and $\gamma_{n,N} = c_{n,m}$ for $N_{m-1} < N \le N_m$. Let $G^\prime \in A_\mathcal{U}(E)^\prime$ be given by
    
    \begin{equation}
        G^\prime(x_\mathcal{U}) = \mathcal{U}-\lim_{N\rightarrow\infty}\frac{1}{N}\sum_{n = 1}^Ng_N^\prime(x_n\gamma_{n,N}) = \mathcal{V}-\lim_{m\rightarrow\infty}\frac{1}{N_m}\sum_{n = 1}^{N_m}f_m^\prime(x_nc_{n,m}).
    \end{equation}
    For $m \ge 1$ and $1 \le w \le N_m$, let $k_{w}(m) \in \{k_{w,i}\}_{i = 1}^{b_w}$ be such that 
    \begin{equation}
        \frac{2k_w(m)}{N_m}+\frac{1}{N_m}\sum_{n = 1}^{N_m-k_w(m)}|\lambda_mc_n-c_{n+k_{w}(m)}| < \delta_{w}.
    \end{equation}
    For $M \in \mathbb{N}$ and $(x_n)_{n = 1}^\infty \in A(E)$ let $x_n(M) = \frac{x_n}{\|x_n\|}\text{min}(\|x_n\|,M)$, let $\lambda =
    \mathcal{V}-\lim_{m\rightarrow\infty}\lambda_m^{-1}$, let $k_w = \mathcal{V}-\lim_{m\rightarrow\infty}k_w(m)$, and observe that

    \begin{alignat*}{2}
        &\left|G^\prime(S_{\mathcal{U}}^{k_w}x(M)_\mathcal{U})-\lambda G^\prime(x(M)_\mathcal{U})\right| = \left|\mathcal{V}-\lim_{m\rightarrow\infty}\frac{1}{N_m}\sum_{n = 1}^{N_m}f_m^\prime((x_{n+k_w,M}-\lambda x_{n,M})c_{n,m})\right|\\
        \le& \mathcal{V}-\lim_{m\rightarrow\infty}\left|\frac{1}{N_m}\sum_{n = 1+k_w(m)}^{N_m}f_m^\prime(x_n(M)(c_{n-k_w(m)}-\lambda_m^{-1} c_{n,m}))+\frac{2Mk_w(m)}{N_m}\right|\\
        \le &\mathcal{V}-\lim_{m\rightarrow\infty}\frac{1}{N_m}\sum_{n = 1}^{N_m-k_w(m)}M|\lambda_mc_n-c_{n+k_w(m)}|+\frac{2Mk_w(m)}{N_m} \le M\delta_w,\text{ hence }\\
         &\sum_{w = 1}^\infty\left|G^\prime(S_{\mathcal{U}}^{k_w}x(M)_\mathcal{U})-\lambda G^\prime(x(M)_\mathcal{U})\right| < \infty.
    \end{alignat*}
    If follows that for any $t,W \in \mathbb{N}$, and any $W < w_1 < w_2 < \cdots < w_t$, and $K_t = \sum_{j = 1}^tk_{w_j}$, we have

    \begin{alignat*}{2}
        &\left|G^\prime(S_{\mathcal{U}}^{K_t}x(M)_\mathcal{U})-\lambda^tG^\prime(x(M)_\mathcal{U})\right|\\
        \le &\sum_{j = 1}^t\left|\lambda^{t-j}G^\prime(S_{\mathcal{U}}^{K_j}x(M)_\mathcal{U})-\lambda^{t-j+1}G^\prime(S_{\mathcal{U}}^{K_{j-1}}x(M)_\mathcal{U})\right|\\
        =&\sum_{j = 1}^t\left|G^\prime(S_{\mathcal{U}}^{k_{w_j}}S_{\mathcal{U}}^{K_{j-1}}x(M)_\mathcal{U})-\lambda G^\prime(S_{\mathcal{U}}^{K_{j-1}}x(M)_\mathcal{U})\right| < M\sum_{w = W}^\infty \delta_w.
    \end{alignat*}
    Let $\mathcal{W} \in \beta\mathbb{N}$ be an idempotent ultrafilter for which $\text{FS}(\{k_w\}_{w = W}^\infty) \in \mathcal{W}$ for all $W \in \mathbb{N}$ (see \cite[Lemma 5.11]{AlgebraInTheSCC}). There exists $\lambda_1 \in \mathbb{S}^1$ for which

    \begin{alignat*}{2}
        &\left|\mathcal{W}-\lim_{k\rightarrow\infty}G^\prime(S_{\mathcal{U}}^ke_\mathcal{U})\right| = \left|\lim_{M\rightarrow\infty}\mathcal{W}-\lim_{k\rightarrow\infty}G^\prime(S_{\mathcal{U}}^ke(M)_\mathcal{U})\right|\\
        =& \left|\lim_{M\rightarrow\infty}\lambda_1G^\prime(e(M)_\mathcal{U})\right| = \left|\lambda_1G^\prime(e_\mathcal{U})\right| = \left|G^\prime(e_\mathcal{U})\right| \ge \epsilon.
    \end{alignat*}
    Since every member of $\mathcal{W}$ is an IP-set (see \cite[Theorem 5.8]{AlgebraInTheSCC}), we have contradicted the fact that $(e_n)_{n = 1}^\infty$ is nearly mildly mixing.
\end{proof}

\begin{remark}\label{RemarkAboutTheDifficultiesOfStrongMixing}
    While we would also like to prove a uniform pointwise ergodic theorem for strongly mixing operators, we do not do so because it is not clear what the statement of the theorem should be. To better understand this, let us recall that for a measure preserving system $\mathcal{X} := (X,\mathscr{B},\mu,\varphi)$, the system $\mathcal{X}$ is ergodic if and only if it has no nontrivial invariant factor, it is weakly mixing if and only if it has no nontrivial Kronecker factor, and it is mildly mixing if and only if it has no nontrivial rigid factor. Parreau has shown a similar result for strongly mixing systems (cf. \cite[Theorem 11]{SpectralTheoryOfDynamicalSystemsOriginal} or \cite{AntiMixingFactor}), i.e., there exists a class of system $\mathcal{S}$ such that the system $\mathcal{X}$ is strongly mixing if and only if it does not contain a factor from the class $\mathcal{S}$. While the existence of the class $\mathcal{S}$ can be used to state an abstract pointwise ergodic theorem for strongly mixing systems, a uniform pointwise ergodic theorem would require a more concrete description of some members of $\mathcal{S}$ in order to apply our compactness arguments about a sequence of local counter examples converging to a global counter example. However, the class $\mathcal{S}$ does not currently have a simple description.
\end{remark}

%%%%%%%%%%%%%%%%%%%%%%%%%%%%%%%%%%%%%%%%%%%%%%%%%%%%%%%%%%%%%%%%%%%%%%%%%%%%%%%%%%%%%%%%%%%%%%%%%%%%%%%%%%%%%%%%%%%%%%%%
\section{Appendix: Comparison of terms}
In this appendix we mention various notions of mixing sequences that have appeared in different parts of the literature. For the sake of presentation, we focus the discussion on variations of ergodic and weakly mixing sequences.

The notion of ergodic, weakly mixing and strongly mixing sequences of vectors in a Hilbert space was introduced by Bergelson and Berend \cite{BerendBergelson}. Mukhamedov \cite{WeaklyMixingSequencesInBanachSpaces} then extended the defintions of ergodic sequences and weakly mixing sequences to apply to sequences in Banach spaces. It is worth noting that weakly mixing and strongly mixing sequences can only exist in an infinite dimensional Hilbert/Banach space, while our almost mixing sequences can exist in $E = \mathbb{C}$. Later, Moreira, Richter, and Robertson \cite[Definition 3.17]{TheErdosSumsetPaper} introduced a notion of weak mixing for certain functions $f:\mathbb{N}\rightarrow\mathbb{C}$, which can also be seen as a notion of weakly mixing sequences of complex numbers, as was done by the second author in \cite{SohailsFirstPointwiseErgodicTheorem}. In \cite[Chapter 2.2]{SohailsPhDThesis} the notion of weakly mixing functions $f:\mathbb{N}\rightarrow\mathbb{C}$ was modified to obtain a new notion of weakly mixing sequences of vectors in a Hilbert space. In order to avoid confusion with the notion of weakly mixing sequences of Berend and Bergelson, this new notion was called nearly weakly mixing, and in \cite[Chapter 2.6]{SohailsPhDThesis} it is shown that every weakly mixing sequences of bounded vectors in a Hilbert space is also a nearly weakly mixing sequences, but the converse need not be true. Interestingly, the analogue of a nearly weakly mixing sequence in the case of ergodicity is called completely ergodic sequence since it is always an ergodic sequence, but the converse need not be true. Now another complication must be pointed out to the reader. The mixing sequences of \cite{BerendBergelson} are bounded sequences, the nearly (completely) mixing sequences $(x_n)_{n = 1}^\infty$ of \cite[Chapter 2]{SohailsPhDThesis} satisfy

\begin{equation}\label{CesaroL2ConditionEquation}
    \limsup_{N\rightarrow\infty}\frac{1}{N}\sum_{n = 1}^N\|x_n\|^2 < \infty,
\end{equation}
and the almost $p$-mixing (fully ergodic)
sequences $(x_n)_{n = 1}^\infty$ of \cite[Chapter 3]{SohailsPhDThesis} (which extends the work of \cite{SohailsFirstPointwiseErgodicTheorem}) satisfy 

\begin{equation}\label{CesaroL1ConditionEquation}
    \limsup_{N\rightarrow\infty}\frac{1}{N}\sum_{n = 1}^N\|x_n\| < \infty.
\end{equation}
Condition \eqref{CesaroL2ConditionEquation} is a natural condition to impose when working with Hilbert spaces, and condition \eqref{CesaroL1ConditionEquation} is a natural condition to impose when working with sequences of the form $(f(\varphi^nx))_{n = 1}^\infty$ with $f \in L^1(X,\mu;\mathcal{H})$. This necessity to work with different classes of sequences is what necessitated different definitions of mixing sequences between Chapter 2 and 3 of \cite{SohailsPhDThesis}. The mixing sequence in \cite[Chapter 3]{SohailsPhDThesis} are called almost $p$-mixing sequences, where $p$ is a filter that corresponds to the level of mixing being considered. If $(x_n)_{n = 1}^\infty$ is a bounded nearly weakly mixing sequence, then it is also almost $\mathcal{D}$-mixing, so in light of the previous discussion every weakly mixing sequence of bounded vectors is also almost $\mathcal{D}$-mixing, but the converse need not be true. If $(x_n)_{n  = 1}^\infty$ is such that

\begin{equation}
    \limsup_{N\rightarrow\infty}\frac{1}{N}\sum_{n = 1}^\infty\|x_n\|^2 \in (0,\infty)\text{ and }\lim_{N\rightarrow\infty}\frac{1}{N}\sum_{n = 1}^N\|x_n\| = 0,
\end{equation}
then $(x_n)_{n = 1}^\infty$ will be almost $\mathcal{D}$-mixing regardless of whether or not it is nearly weakly mixing (see \cite[Remark 2.3.7]{SohailsPhDThesis}). Similarly, for bounded sequences of vectors the notion of complete ergodicity is strictly stronger than the notion of full ergodicity, which is strictly stronger than the notion of ergodicity. 

It is natural to ask if the definition of almost weakly mixing sequence of vectors in a Banach space that we have given in Section \ref{UltraproductsSubsection} reduces to the definition of almost $\mathcal{D}$-mixing sequences in a Hilbert space if our Banach is also a Hilbert space $\mathcal{H}$. To this end, we present an equivalent definition of almost $\mathcal{D}$-mixing sequences that is suited to our needs.

\begin{definition}
    Suppose that $\mathcal{H}$ is a Hilbert space and $(x_n)_{n = 1}^\infty \subseteq \mathcal{H}$ satisfies Equation \eqref{CesaroL1ConditionEquation}.\footnote{It is also worth pointing out here that satisfying Equation \eqref{CesaroL1ConditionEquation} is equivalent to being an element of $ces_\infty(\mathcal{H})$. Furthermore, almost weakly mixing sequences are (by definition) elements of $A(\mathcal{H}) \subsetneq ces_\infty(\mathcal{H})$.} The sequence $(x_n)_{n = 1}^\infty$ is \textbf{almost $\mathcal{D}$-mixing} if for any bounded sequence $(y_n)_{n = 1}^\infty \subseteq \mathcal{H}$ and any ultrafilter $\mathcal{U} \in \beta\mathbb{N}^*$ we have

    \begin{equation}
        \mathcal{D}-\lim_{h\rightarrow\infty}\mathcal{U}-\lim_{N\rightarrow\infty}\frac{1}{N}\sum_{n = 1}^N\langle x_{n+h},y_n\rangle = 0.
    \end{equation}
\end{definition}

To see that any almost weakly mixing sequence is also an almost $\mathcal{D}$-mixing sequence, it suffices to observe that for any bounded sequence $\vec{y} := (y_n)_{n = 1}^\infty \subseteq \mathcal{H}$ and any ultrafilter $\mathcal{U} \in \beta\mathbb{N}^*$ there is a functional $g_{\vec{y},\mathcal{U}} \in A_{\mathcal{U}}(\mathcal{H})^\prime$ given by

\begin{equation}
    \langle i((x_n)_{n = 1}^\infty),g_{\vec{y},\mathcal{U}}\rangle_\mathcal{U} = \mathcal{U}-\lim_{N\rightarrow\infty}\frac{1}{N}\sum_{n = 1}^N\langle x_n,y_n\rangle.
\end{equation}
We currently do not know whether every almost $\mathcal{D}$-mixing element of $A(\mathcal{H})$ is also almost weakly mixing, because we do not know whether or not every element of $A_\mathcal{U}(\mathcal{H})^\prime$ has the form $g_{\vec{y},\mathcal{U}}$.\\

\noindent\textbf{Acknowledgements:} The second author acknowledges being supported by grant
2019/34/E/ST1$\allowbreak$/00082 for the project “Set theoretic methods in dynamics and number theory,” NCN (The
National Science Centre of Poland), and the first author was supported by the same grant for a 3 week research visit. The second author would like to thank Uta Freiberg for the chance to visit the TU Chemnitz with a one week research stay. Both authors would like to express our gratitude to the TU Chemnitz and the University of Adam Mickiewicz in Poznań for enabling our research stays and facilitating this collaboration. The authors would also like to thank Thomas Kalmes, Noa Bihlmaier and Pablo Lummerzheim for helpful discussions. Lastly, 
the authors would like to thank the referees for their careful reading of this article, and their comments that greatly improved the exposition.\\

\noindent\textbf{Data availability} Data sharing not applicable to this article as no datasets
were generated or analysed during the current study.\\

\noindent\textbf{Declarations}
\textbf{Conflict of interest} The authors have no relevant financial or non-financial
interests to declare.

\bibliographystyle{abbrv}
\begin{center}
	\bibliography{references}

\begin{thebibliography}{10}

\bibitem{CesaroFunctionSpaces}
S.~V. Astashkin and L.~Maligranda.
\newblock Structure of {C}es\`aro function spaces: a survey.
\newblock In {\em Function spaces {X}}, volume 102 of {\em Banach Center
  Publ.}, pages 13--40. Polish Acad. Sci. Inst. Math., Warsaw, 2014.

\bibitem{BerendBergelson}
D.~Berend and V.~Bergelson.
\newblock Mixing sequences in {H}ilbert spaces.
\newblock {\em Proc. Amer. Math. Soc.}, 98(2):239--246, 1986.

\bibitem{Birkhoff'sErgodicTheorem}
G.~D. Birkhoff.
\newblock Proof of the ergodic theorem.
\newblock {\em Proc. Natl. Acad. Sci. USA}, 17:656--660, 1931.

\bibitem{DoubleRecurrenceAndASConvergence}
J.~Bourgain.
\newblock Double recurrence and almost sure convergence.
\newblock {\em J. Reine Angew. Math.}, 404:140--161, 1990.

\bibitem{ChaconErgodicTheorem}
R.~V. Chacon.
\newblock An ergodic theorem for operators satisfying norm conditions.
\newblock {\em J. Math. Mech.}, 11:165--172, 1962.

\bibitem{MorePointwiseErgodicTheoremsForDSOperators}
V.~Chilin, D.~\c{C}\"{o}mez, and S.~Litvinov.
\newblock Individual ergodic theorems for infinite measure.
\newblock {\em Colloq. Math.}, 167(2):219--238, 2022.

\bibitem{DSValiditySpace}
V.~Chilin and S.~Litvinov.
\newblock The validity space of {D}unford-{S}chwartz pointwise ergodic theorem.
\newblock {\em J. Math. Anal. Appl.}, 461(1):234--247, 2018.

\bibitem{AUWWforDSOperators}
V.~Chilin and S.~Litvinov.
\newblock Almost uniform convergence in the {W}iener-{W}intner ergodic theorem.
\newblock {\em Studia Math.}, 259(3):327--338, 2021.

\bibitem{DoobPointwiseTheorem}
J.~L. Doob.
\newblock Asymptotic properties of {M}arkoff transition prababilities.
\newblock {\em Trans. Amer. Math. Soc.}, 63:393--421, 1948.

\bibitem{DSOperators}
N.~Dunford and J.~T. Schwartz.
\newblock Convergence almost everywhere of operator averages.
\newblock {\em J. Rational Mech. Anal.}, 5:129--178, 1956.

\bibitem{stability}
T.~Eisner.
\newblock {\em Stability of Operators and Operator Semigroups}.
\newblock Operator Theory: Advances and Applications, Vol. 209. Birkhäuser
  Verlag, 2010.

\bibitem{OTAoET}
T.~Eisner, B.~Farkas, M.~Haase, and R.~Nagel.
\newblock {\em Operator theoretic aspects of ergodic theory}, volume 272 of
  {\em Graduate Texts in Mathematics}.
\newblock Springer, Cham, 2015.

\bibitem{UniformConvergenceOfTwistedErgodicAverages}
T.~Eisner and B.~Krause.
\newblock ({U}niform) convergence of twisted ergodic averages.
\newblock {\em Ergodic Theory Dynam. Systems}, 36(7):2172--2202, 2016.

\bibitem{OnModulatedErgodicTheorems}
T.~Eisner and M.~Lin.
\newblock On modulated ergodic theorems.
\newblock {\em J. Nonlinear Var. Anal.}, 2(2):131--154, 2018.

\bibitem{UniformNilsequenceWW}
T.~Eisner and P.~Zorin-Kranich.
\newblock Uniformity in the {W}iener-{W}intner theorem for nilsequences.
\newblock {\em Discrete Contin. Dyn. Syst.}, 33(8):3497--3516, 2013.

\bibitem{OneParameterOperatorSemigroups}
E.~Y. Emel’yanov.
\newblock {\em Non-spectral asymptotic analysis of one-parameter operator
  semigroups}, volume 173 of {\em Operator Theory: Advances and Applications}.
\newblock Birkhäuser Verlag, Basel, 2007.

\bibitem{SohailsFirstPointwiseErgodicTheorem}
S.~Farhangi.
\newblock Pointwise ergodic theorems for higher levels of mixing.
\newblock {\em Studia Math.}, 261(3):329--344, 2021.

\bibitem{SohailsPhDThesis}
S.~Farhangi.
\newblock {\em Topics in ergodic theory and ramsey theory}.
\newblock {PhD} dissertation, the Ohio State University, 2022.

\bibitem{UniformPolynomialWW}
N.~Frantzikinakis.
\newblock Uniformity in the polynomial {W}iener-{W}intner theorem.
\newblock {\em Ergodic Theory Dynam. Systems}, 26(4):1061--1071, 2006.

\bibitem{UltraproductsInBanachSpaceTheory}
S.~Heinrich.
\newblock Ultraproducts in {B}anach space theory.
\newblock {\em J. Reine Angew. Math.}, 313:72--104, 1980.

\bibitem{AlgebraInTheSCC}
N.~Hindman and D.~Strauss.
\newblock {\em Algebra in the {S}tone-\v{C}ech compactification: Theory and
  applications}.
\newblock De Gruyter Textbook. Walter de Gruyter \& Co., Berlin, second revised
  and extended edition, 2012.

\bibitem{NoncommutativeMultiParameterWWTheorem}
G.~Hong and M.~Sun.
\newblock Noncommutative multi-parameter {W}iener-{W}intner type ergodic
  theorem.
\newblock {\em J. Funct. Anal.}, 275(5):1100--1137, 2018.

\bibitem{HopfErgodicTheorem}
E.~Hopf.
\newblock The general temporally discrete {M}arkoff process.
\newblock {\em J. Rational Mech. Anal.}, 3:13--45, 1954.

\bibitem{AnalysisInBanachSpacesVolume1}
T.~Hyt\"{o}nen, J.~van Neerven, M.~Veraar, and L.~Weis.
\newblock {\em Analysis in {B}anach spaces. {V}ol. {I}. {M}artingales and
  {L}ittlewood-{P}aley theory}, volume~63 of {\em Ergebnisse der Mathematik und
  ihrer Grenzgebiete. 3. Folge. A Series of Modern Surveys in Mathematics
  [Results in Mathematics and Related Areas. 3rd Series. A Series of Modern
  Surveys in Mathematics]}.
\newblock Springer, Cham, 2016.

\bibitem{L1PETFailsForMostInvertibleIsometries}
A.~Ionescu~Tulcea.
\newblock On the category of certain classes of transformations in ergodic
  theory.
\newblock {\em Trans. Amer. Math. Soc.}, 114:261--279, 1965.

\bibitem{BirkhoffForMarkoffProcesses}
S.~Kakutani.
\newblock Ergodic theorems and the {M}arkoff process with a stable
  distribution.
\newblock {\em Proc. Imp. Acad. Tokyo}, 16:49--54, 1940.

\bibitem{ErgodicPropertiesOfLampertiOperators}
C.~H. Kan.
\newblock Ergodic properties of {L}amperti operators.
\newblock {\em Canadian J. Math.}, 30(6):1206--1214, 1978.

\bibitem{UKErgodicThms}
U.~Krengel.
\newblock {\em Ergodic theorems}, volume~6 of {\em De Gruyter Studies in
  Mathematics}.
\newblock Walter de Gruyter \& Co., Berlin, 1985.
\newblock With a supplement by Antoine Brunel.

\bibitem{OnTheIsometriesOfCertainFunctionSpaces}
J.~Lamperti.
\newblock On the isometries of certain function-spaces.
\newblock {\em Pacific J. Math.}, 8:459--466, 1958.

\bibitem{SpectralTheoryOfDynamicalSystemsOriginal}
M.~Lema\'{n}czyk.
\newblock Spectral theory of dynamical systems.
\newblock In {\em Mathematics of complexity and dynamical systems. {V}ols.
  1--3}, pages 1618--1638. Springer, New York, 2012.

\bibitem{OriginalUniformPolynomialWW}
E.~Lesigne.
\newblock Un th\'{e}or\`eme de disjonction de syst\`emes dynamiques et une
  g\'{e}n\'{e}ralisation du th\'{e}or\`eme ergodique de {W}iener-{W}intner.
\newblock {\em Ergodic Theory Dynam. Systems}, 10(3):513--521, 1990.

\bibitem{PolynomialWW}
E.~Lesigne.
\newblock Spectre quasi-discret et th\'{e}or\`eme ergodique de
  {W}iener-{W}intner pour les polyn\^{o}mes.
\newblock {\em Ergodic Theory Dynam. Systems}, 13(4):767--784, 1993.

\bibitem{TheErdosSumsetPaper}
J.~Moreira, F.~K. Richter, and D.~Robertson.
\newblock A proof of a sumset conjecture of {E}rdos.
\newblock {\em Ann. of Math. (2)}, 189(2):605--652, 2019.

\bibitem{WeaklyMixingSequencesInBanachSpaces}
F.~Mukhamedov.
\newblock On tensor products of weak mixing vector sequences and their
  applications to uniquely {$E$}-weak mixing {$C^\ast$}-dynamical systems.
\newblock {\em Bull. Aust. Math. Soc.}, 85(1):46--59, 2012.

\bibitem{NonCommutativeDSWienerWintnerTheorem}
M.~O'Brien.
\newblock Noncommutative {W}iener-{W}intner type ergodic theorems.
\newblock {\em Studia Math.}, 269(2):209--239, 2023.

\bibitem{AntiMixingFactor}
F.~Parreau.
\newblock Facteurs disjoints des transformations melangeantes.
\newblock {\em arXiv:2307.01562}, 2023.

\bibitem{Wiener-WintnerTheorem}
N.~Wiener and A.~Wintner.
\newblock Harmonic analysis and ergodic theory.
\newblock {\em Amer. J. Math.}, 63:415--426, 1941.

\bibitem{MostGeneralVectorValuedErgodicTheorem}
T.~Yoshimoto.
\newblock Vector-valued ergodic theorems for operators satisfying norm
  conditions.
\newblock {\em Pacific J. Math.}, 85(2):485--499, 1979.

\end{thebibliography}
\end{center}
\end{document}